\documentclass[mathpazo]{csam}
\usepackage[margin=1in]{geometry}
\usepackage{amsmath,amsfonts,amssymb,amsthm,bbold}
\usepackage{graphicx,bm,subfig,cases}
\usepackage{booktabs}
\usepackage[export]{adjustbox}
\usepackage[update,prepend,suffix=]{epstopdf} 
\graphicspath{{figure/}}
\usepackage[referable]{threeparttablex}
\usepackage{tikz,tikz-3dplot}
\usetikzlibrary{arrows,decorations.pathreplacing,shapes,positioning,calc,fit,intersections}

\newcommand{\bL}{\bm L}
\newcommand{\bx}{\bm x}
\newcommand{\bH}{\mathbf H}
\newcommand{\bJ}{\mathbf J}
\newcommand{\jj}{j=1,\cdots,J}
\newcommand{\qq}{q=1,\cdots,Q_j}
\newcommand{\dd}{\mathrm d}

\DeclareMathOperator{\vech}{vech}

\DeclareMathOperator{\st}{s.t.}

\def\Xint#1{\mathchoice
{\XXint\displaystyle\textstyle{#1}}%
{\XXint\textstyle\scriptstyle{#1}}%
{\XXint\scriptstyle\scriptscriptstyle{#1}}%
{\XXint\scriptscriptstyle\scriptscriptstyle{#1}}%
\!\int}
\def\XXint#1#2#3{{\setbox0=\hbox{$#1{#2#3}{\int}$ }
\vcenter{\hbox{$#2#3$ }}\kern-.6\wd0}}
\def\aint{\Xint-}
\begin{document}

\title{{\bf\large An Integrated Quadratic Reconstruction for Finite
    Volume Schemes to Scalar Conservation Laws in Multiple Dimensions}} 

\author[L. Chen et~al.]{Li Chen\affil{1}, Ruo Li\affil{2}\corrauth, Feng Yang\affil{1}}

\emails{{\tt cheney@pku.edu.cn}(Li Chen), {\tt rli@math.pku.edu.cn}(Ruo Li),  {\tt f\_yang@pku.edu.cn}(Feng Yang)}

\address{\affilnum{1}\ School of Mathematical Sciences,
	Peking University, Beijing, China\\
	\affilnum{2}\ HEDPS \& CAPT, LMAM \& 
	School of Mathematical Sciences,
	Peking University, Beijing, China}

\begin{abstract}
  We proposed a piecewise quadratic reconstruction method in
  multiple dimensions, which is in an integrated style, for finite
  volume schemes to scalar conservation laws. This integrated
  quadratic reconstruction is parameter-free and applicable on
  flexible grids. We show that the finite volume schemes with the new
  reconstruction satisfy a local maximum principle with properly setup
  on time steplength. Numerical examples are presented to show that
  the proposed scheme attains a third-order accuracy for smooth
  solutions in both 2D and 3D cases. It is indicated by numerical
  results that the local maximum principle is helpful to prevent
  overshoots in numerical solutions.
\end{abstract}


	
\maketitle
\section{Introduction}

The study of robust, accurate, and efficient finite volume schemes for
conservation laws is an active research area in computational fluid dynamics. It
was noted that higher-order finite volume methods have been shown to be more
efficient than second-order methods \cite{Michalak2009}. The key element in the
reconstruction procedures of high-order schemes is suppressing non-physical
oscillations near discontinuities, while achieving high-order accuracy in smooth
regions. One of the pioneering work in this area is the finite volume scheme
based on the $k$-exact reconstruction, first proposed by Barth and Fredrichson
\cite{Barth1990} and later extended to the cell-centered finite volume scheme by
Mitchell and Walters \cite{Mitchell1993}.  For more recent work on the use of
$k$-exact reconstruction to attain high-order accuracy, we refer the reader to
\cite{Barth1993, Ollivier-Gooch2009, Michalak2009, Li2012, Hu2016} for instance.
The hierarchical reconstruction strategies of Liu \textit{et al.}
\cite{Liu2007} were also used to achieve higher-order accuracy \cite{Xu2011,
Hu2010}, where the information is recomputed level by level from the highest
order terms to the lowest order terms with certain non-oscillatory method.
Other type of high-order finite volume schemes includes the WENO scheme
\cite{Liu1994}.  Although the implementation of WENO scheme is comparatively
complicated on unstructured meshes due to the needs of identifying several
candidate stencils and performing a reconstruction on each stencil
\cite{Li2012}, it has been successfully applied on the unstructured meshes for
both two-dimensional triangulations \cite{Friedrich1998, Hu1999, Dumbser2007,
Dumbser2007a, Titarev2010, Liu2013} and three-dimensional triangulations
\cite{Zhang2009, Tsoutsanis2011}.  Most of these schemes do not lead to a strict
maximum principle, while they are essentially non-oscillatory \cite{Shu2017}.
Actually, the reconstruction procedure for maximum-principle-satisfying
second-order schemes are relatively mature \cite{Barth1989, Durlofsky1992,
Liu1993, Batten1996, Hubbard1999, Park2010}, while there are few
maximum-principle-satisfying reconstruction approaches for higher-order finite
volume schemes on unstructured meshes.

Limiting to scalar conservation laws, a quadratic reconstruction for finite
volume schemes applicable on 2D and 3D unstructured meshes is developed in this
paper. The construction is a further exploration of the integrated linear
reconstruction (ILR) in \cite{Chen2016, Chen2018}, where the coefficients of
reconstructed polynomial are embedded in an optimization problem. It is
appealing for us to generalize the optimization-based constructions therein to
an integrated quadratic reconstruction (IQR) such that the scheme achieves a
third-order accuracy while satisfying a local maximum principle. It was pointed
out in \cite{Zhang2011a} that the scheme satisfying the standard local maximum
principle is at most second-order accurate around extrema. To achieve higher
than second-order accuracy, high-order information of the exact solution has to
be taken into account in the definition of local maximum principle
\cite{Zhang2011a, Xu2017}. Sanders \cite{Sanders1988} suggested to measure the
total variation of approximation polynomials. Liu \textit{et al.} \cite{Liu1996}
constructed a third-order non-oscillatory scheme by controlling the number of
extrema and the range of the reconstructed polynomials. Zhang \textit{et al.}
constructed a genuinely high-order maximum-principle-satisfying finite volume
schemes for multi-dimensional nonlinear scalar conservation laws on both
rectangular meshes \cite{Zhang2010a} and triangular meshes \cite{Zhang2012} by
limiting the reconstructed polynomials around cell averages. The flux limiting
technique developed by Christlieb \textit{et al.} \cite{Christlieb2015} is
another family of maximum-principle-satisfying methods on unstructured meshes.
In our scheme, it is proposed that the extrema of numerical solutions are
measured by extrema of polynomial on a cluster of points, following the
technique of Zhang \textit{et al.} \cite{Zhang2010a, Zhang2012}. To overcome the
difficulty of loss of high-order information in cell averages, besides the cell
averages at current time level, we utilize the reconstruction polynomials at
previous time step. This idea is based on the wave propagation nature of
conservation laws. Since the solution value at $({\bf x}, t)$ can be tracked
back to a point in the ball around ${\bf x}$ with radius as $v \Delta t$ at time
$t - \Delta t$, while $v$ is local wave speed, it is reasonable for us to use
the value in this ball at previous time step in the reconstruction. It is shown
that this may lead to a third-order numerical scheme, meanwhile a local maximum
principle is satisfied. An advantage of the new reconstruction is that there is
no artificial parameter at all, which makes the algorithm robust and independent
of the problem.

The rest of the paper is organized as follows. In Section 2, we describe the
integrated quadratic reconstruction based on solving a series of quadratic
programming problems. Section 3 is devoted to the discussion of the order of
accuracy and maximum principle for scalar conservation laws.  Numerical results
are given to demonstrate the stability and accuracy of the proposed scheme in
Section 4.  Finally, a short conclusion is drawn in Section 5.


\section{Numerical Scheme}

Let us consider a \emph{scalar} hyperbolic conservation law on a
$d$-dimensional domain $\Omega$, $d = 2, 3$, as
\begin{equation}
  \dfrac{\partial u}{\partial t}+\nabla\cdot \bm F(u)= 0,
  \label{eq:hcl}
\end{equation}
together with appropriate boundary condition and initial value $u(\cdot,0)$. The
computational domain $\Omega$ is triangulated into a grid, either structured or
unstructured, denoted by $\mathcal{T}$.  For an arbitrary cell $T_0 \in \mathcal
T$ referred as a control volume for finite volume method, let $e_j$ be the facet
of $T_0$ shared by $T_0$ and its von Neumann neighbor $T_j$, and $\bm n_j$ be
the unit outer normal of $e_j$~$(\jj)$. The finite volume discretization for
\eqref{eq:hcl} is then formulated as
\begin{equation}
\dfrac{u_0^{n+1}-u_0^{n}}{\Delta t_n} + \dfrac1{|T_0|} \sum_{j=1}^J\sum_{q=1}^{Q_j}
w_{jq}\mathcal F(v_{h,0}^n(\bm z_{jq}),v_{h,j}^n(\bm z_{jq});\bm n_{j}) |e_{j}| =  0.
\label{eq:fvm-euler}
\end{equation}
Here $u_0^{n}$ approximates the cell average of the solution $u$ on $T_0$ at
$n$-th time level $t_n$, i.e.
\[
u_0^n \approx \left. \Pi u(\cdot, t_n) \right|_{T_0},
\]
where $\Pi$ is the piecewise constant projection defined by
\[
\left. \Pi w \right|_{T_0} = \aint_{T_0} w(\bm x)\mathrm
d\bm x := \dfrac{1}{|T_0|}\int_{T_0} w(\bm x)\mathrm d\bm x, \quad
\forall w \in L^1(\Omega).
\] 
The point $\bm z_{jq}$ is the $q$-th quadrature point on the facet $e_j$ with
weight $w_{jq}$~($\qq,\jj$), the function $v_{h,0}^n(\bm x)$ is a reconstructed
polynomial computed from the patch of cell $T_0$, the function
$v_{h,j}^n(\bm x)$ ($j=1, \cdots, J$) is a reconstructed
polynomial computed from the patch of cell $T_j$, and $\mathcal F(u,v;\bm n)$ is
a numerical flux, such as the Lax-Friedrichs flux
\begin{equation}
\mathcal F(u,v;\bm n)
=\dfrac{1}{2}(\bm F(u)+\bm F(v))\cdot\bm n-\dfrac{1}{2}a(v-u),
\label{eq:lfflux}
\end{equation}
where $a=\sup\limits_{u,\bm n} |\bm F'(u)\cdot \bm n|$ represents the maximal
characteristic speed.

Denote the piecewise constant approximation of $u(\bx, t_n)$ to be $u_h^n(\bx)$,
which takes $u_0^n$ as its value on $T_0$.  In this paper we will focus on
constructing a quadratic polynomial $v_{h,0}^n(\bx)$ on each control volume $T_0$.
And the resulting piecewise quadratic function on the whole domain $\Omega$ is
denoted by $v_h^n(\bx)$.  Classical patch reconstruction algorithms in the
literature directly give $v_h^n(\bx)$ from $u_h^n(\bx)$, while the integrated
quadratic reconstruction requires additional information.  Precisely, we may
formulate our reconstruction as an operator $\mathcal{R}_h$ 
\[
  v_h^n(\bx) = \mathcal{R}_h [u_h^n, v_h^{n-1} ] (\bx).
\]
That is to say, the function $v_h^n$ depends on not only its piecewise constant
counterpart $u_h^n$, but also the previous reconstruction $v_h^{n-1}$.
Basically, the operator $\mathcal{R}_h$ accepts two functions as its arguments:
the first function is a piecewise constant function on $\mathcal{T}$, and the
second function is a piecewise continuous function on $\mathcal{T}$.  With the
introduction of the operator $\mathcal{R}_h$, the numerical scheme
\eqref{eq:fvm-euler} can be formally identified as
\begin{equation}
u_h^{n+1}=u_h^{n}+\Delta t_n \mathcal L(v_h^{n}), \quad
v_h^{n}=\mathcal{R}_h [u_h^n,v_h^{n-1}],
\label{eq:firstorderscheme}
\end{equation}
where $\mathcal L$ is the operator defined through \eqref{eq:fvm-euler}.

For the initial level $n=0$, we directly take $u_h^0(\bx)$ to be the piecewise
constant projection of $u(\bx,0)$ on $\mathcal{T}$ and $v_h^{-1}(\bx) =
u(\bx,0)$, saying
\begin{equation}
v_h^0(\bx) = \mathcal{R}_h [\Pi u(\cdot,0), u(\cdot,0) ] (\bx),
\label{eq:initial-value}
\end{equation}
to bootstrap the computation. And we note that \eqref{eq:initial-value} actually
defines a mapping from a continuous function $w \in C(\Omega) \bigcap
L^1(\Omega)$ to a piecewise quadratic function on $\mathcal{T}$.  We denote this
mapping again by $\mathcal{R}_h$ 
\[
\mathcal{R}_h [w](\bx) := \mathcal{R}_h[\Pi w, w](\bx),
\]
for convenience. Therefore, we need only to specify $\mathcal{R}_h$ to close the
scheme \eqref{eq:fvm-euler}. Below we describe the procedure to specify
$v_h^n(\bx)$ on a single cell $T_0$ using $u_h^n(\bx)$ and $v_h^{n-1}(\bx)$. 

The reconstructed quadratic function on cell $T_0$ can be formulated as
\begin{equation}
v_{h,0}^{n}(\bm x)=u_0^{n}+\bL\cdot(\bm x-\bm x_0)
+
\dfrac{1}{2}\bH:((\bm x-\bm x_0)\otimes(\bm x-\bm x_0)
-\bJ_0),
\label{eq:recon}
\end{equation}
where the operator $\otimes$ denotes the tensor product of vectors,
and the operator $:$ denotes the inner-product of high order tensor.
The vector $\bL$ and the matrix $\bH$ are
\[
\bL=
\begin{bmatrix}
L_1 \\ L_2 \\ \cdots \\ L_d
\end{bmatrix}\quad \text{~and~} \quad
\bH=
\begin{bmatrix}
H_{11} & H_{12} & \cdots &H_{1d} \\
H_{21} & H_{22} & \cdots &H_{2d} \\
\vdots & \vdots & \ddots & \vdots \\
H_{d1} & H_{d2} & \cdots &H_{dd} 
\end{bmatrix},
\]
approximates respectively the gradient $\nabla u$ and the Hessian
$\nabla^2 u$ near the centroid $\bx_0$ of cell $T_0$, and $\bJ_0$
represents the second moments of the cell $T_0$
\[
\bJ_0=\aint_{T_0} (\bm x-\bm x_0)\otimes(\bm x-\bm x_0)\mathrm d\bm x,
\]
which depends on the geometry of control volume $T_0$ only. In Table
\ref{tab:sms} we list the second moments of several geometric shapes widely used
in the mesh triangulation.  Note that the quadratic polynomial 
\eqref{eq:recon} automatically satisfies the conservation property, i.e.
\begin{equation*}
\aint_{T_0} v_{h,0}^{n}(\bm x)\mathrm d\bm x=u_0^{n}.
\end{equation*}

\begin{table}[htbp]
   	\centering	
   	\caption{Geometry parameters of several control volumes.}
   	\label{tab:sms}	
   \begin{threeparttable}
   	\scriptsize
	\begin{tabular}{lcll}
		\toprule
		& \bf Geometry \tnote{a} & \bf Second moments \tnote{b} 
        & $\nu$\\
		\midrule
		\bf Rectangle &
		\includegraphics[valign=M]{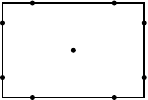} &
		$
		\displaystyle
		\mathbf J=
		\dfrac{1}{12}
		\begin{bmatrix}
		{l_x^2} & 0\\
		0 & {l_y^2}
		\end{bmatrix}
		$
		& $\dfrac{1}{16}$
		\\[10mm]
		\bf Triangle &
		\includegraphics[valign=M]{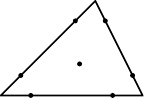} &
		$\displaystyle
		\bJ=\dfrac{1}{36}\sum_{1\le i<j\le 3}
		\overrightarrow{P_iP_j}\otimes\overrightarrow{P_iP_j}
		$
        & $\dfrac{1}{12}$
		\\[10mm]
		\bf Cuboid &
		\includegraphics[valign=M]{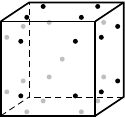} &
		$
		\displaystyle
		\mathbf J=
		\dfrac{1}{12}
		\begin{bmatrix}
		{l_x^2} & 0& 0\\
		0 & {l_y^2} & 0\\
		0 & 0 & {l_z^2}
		\end{bmatrix}
		$
        & $\dfrac{1}{30}$
		\\[10mm]
		\bf Tetrahedron &
		\includegraphics[valign=M]{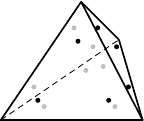}   &
		$\displaystyle
		\bJ=\dfrac{1}{80}\sum_{1\le i<j\le 4}
		\overrightarrow{P_iP_j}\otimes\overrightarrow{P_iP_j}
		$
        & $\dfrac{1}{20}$
		\\ \bottomrule
	\end{tabular}
	\begin{tablenotes}
		\scriptsize
          \item[a] For segmental or rectangular facets, the quadrature points
          are the Gaussian points, while for triangular facets, the barycentric
          coordinates of three quadrature points are $(2/3,1/6,1/6)$,
          $(1/6,2/3,1/6)$ and $(1/6,1/6,2/3)$ respectively.
          \item[b] $l$'s denote the dimensions of the control volume, and
          $P_i$'s denote the vertices of the control volume.
	\end{tablenotes}
\end{threeparttable}   
\end{table}

To suppress numerical oscillations we follow the same basic outline as
traditional second-order limiters, namely, limiting the values on the quadrature
points $\{\bm z_{jq}\}$.  Note that higher-order reconstructions admit local
extrema within cells, in contrast to linear reconstructions.  Therefore, to
improve the restriction within cells, we also examine the value at the centroid
$\bx_0$.  For convenience, we define a cluster of collocation points associated
to a given cell $T_0$ by
\[
Z_0=\{\bm z_{jq}|\qq,\jj\}\cup \{\bx_0\},
\]
which consists of all quadrature points on the cell faces along with the
centroid of the whole cell (see the geometry column in Table \ref{tab:sms}).
Now introduce an objective function depending on the parameters 
$\bL$ and $\bH$ in the expression \eqref{eq:recon} of $v_{h,0}^n(\bx)$ as
\begin{equation}
\delta (\bL,\bH)=\sum_{i\in S} (\overline u_i^{n}-u_i^{n})^2,\quad 
\overline u_i^{n}=\aint_{T_i} v_{h,0}^{n}(\bx)\mathrm d\bx,
\label{eq:objfun}
\end{equation}
which is the sum of squared residuals of mean values of $v_{h,0}^n$ from $u_h^n$ on
the \emph{Moore neighbors} $\{T_i\}_{i\in S}$, namely, those cells sharing at
least one common vertex with $T_0$ (see Fig. \ref{fig:labeln} for several
examples).  Now we are ready to raise the following optimization problem:
\[
\begin{array}{rl}
\min & \delta (\bL,\bH) \\
\mathrm{s.t.} & \text{\eqref{eq:cons} is fulfilled.}
\end{array}
\]
The constraints are some double inequality constraints on the cluster $Z_0$
\begin{subequations}
\begin{equation}
m_{0j}^{n}\le v_{h,0}^{n}(\bm z_{jq})\le M_{0j}^{n},\quad \qq,\jj,
\end{equation}
\begin{equation}
m_{00}^{n}\le v_{h,0}^{n}(\bx_0)\le M_{00}^{n},
\hphantom{\quad \jj,\qq,}
\end{equation}
\label{eq:cons}
\end{subequations}
and the lower and upper bounds in these inequalities are given by
\begin{subequations}
\begin{equation}
\begin{split}
&m_{0j}^{n}=\min\left\{\min_{\bm z\in Z_0}{v^{n-1}_{h,0}(\bm z)},
                \min_{\bm z\in Z_j}{v^{n-1}_{h,j}(\bm z)},
                u_0^{n},u_j^{n}\right\},\\
&M_{0j}^{n}=\max\left\{\max_{\bm z\in Z_0}{v^{n-1}_{h,0}(\bm z)},
               \max_{\bm z\in Z_j}{v^{n-1}_{h,j}(\bm z)},
                u_0^{n},u_j^{n}\right\},\quad \jj,
\end{split}
\end{equation}            
\begin{equation}          
m_{00}^{n}=\min\left\{\min_{\bm z\in Z_0}{v^{n-1}_{h,0}(\bm
    z)},u_0^{n}\right\}, \quad
M_{00}^{n}=\max\left\{\max_{\bm z\in Z_0}{v^{n-1}_{h,0}(\bm 
    z)},u_0^{n}\right\}.
\end{equation}
\label{eq:bounds}
\end{subequations}

\begin{remark}
  It is clear that (\ref{eq:cons}a) is to restrict the value on the cell face
  and (\ref{eq:cons}b) is to restrict the value in the interior of the cell.
  The expression \eqref{eq:bounds} is a prediction based on the wave propagation
  nature for scalar conservation laws.
\end{remark}
\begin{remark}
  A simple observation is that $\mathcal R_h[u]=u$ if $u$ is a quadratic
  polynomial.  Indeed, the linear and quadratic coefficients of $u$ would
  definitely minimize the objective function \eqref{eq:objfun} and satisfy all
  the constraints \eqref{eq:cons}.
\end{remark}

Now we express the optimization problem in
a compact form. Rewrite \eqref{eq:recon} as
\[
\begin{split}
v^{n}_{h,0}(\bm x)&=u_0^{n}+(\bx-\bx_i+\bx_i-\bx_0)\cdot\bL+\dfrac{1}{2}
((\bx_i-\bx_0)\otimes(\bx_i-\bx_0)+(\bx-\bx_i)\otimes(\bx-\bx_i)\\
&+(\bx-\bx_i)\otimes(\bx_i-\bx_0)
+(\bx_i-\bx_0)\otimes(\bx-\bx_i)
-\bJ_0):\bH,
\end{split}
\]
then the integral average of $v_{h,0}^{n}$ on the cell $T_i$ is found to be
\[
\overline u_i^{n}=\aint_{T_i} v_{h,0}^{n}(\bm x)\mathrm d\bx 
=u_0^{n}+\bm r_i\cdot\bL
+\dfrac{1}{2}(\bm r_i\otimes\bm r_i+\bJ_i-\bJ_0):\bH,
\]
where $\bm r_i=\bm x_i-\bm x_0$.  Denote the \emph{half-vectorization} of a
symmetric matrix $\mathbf A=(A_{ij})_{d\times d}$ by vectorizing its lower
triangular part, namely,
\[
\begin{array}{rccccccl}
\vech(\mathbf A)=\big[&A_{11},&A_{21},&\cdots,&A_{d1},&&\\
&&A_{22},&\cdots,&A_{d2},&&\\
&&&&\cdots,&&\\
&&&& A_{dd}&\big]^\top\in\mathbb R^{d(d+1)/2}.&
\end{array}
\]
Then we have the following compact form for $\overline u_i^n$:
\[
  \overline u_i^{n}= u_0^{n}+\bm s_i^\top\bm\varphi,
\]
where the vectors
\begin{equation}
\bm s_i=
\begin{bmatrix}
\bm r_i/h \\ \vech(\bm r_i\otimes \bm r_i+\bJ_i-\bJ_0)/h^2
\end{bmatrix},
\label{eq:rs}
\end{equation}
and
\[
 \begin{array}{rccccccl}
 \bm\varphi=\big[&hL_1,&hL_2,&\cdots,&hL_d,&&\\
 &h^2H_{11}/2,&h^2H_{21},&\cdots,&h^2H_{d1},&&\\
 &&h^2H_{22}/2,&\cdots,&h^2H_{d2},&&\\
 &&&&\cdots,&&\\
 &&&& h^2H_{dd}/2&\big]^\top\in\mathbb R^{d(d+3)/2}.&
 \end{array}
 \]
Here $h$ is a reference length, such as the mesh size of current cell.
Inserting the compact form of $\overline u_i^n$ into the objective function
\eqref{eq:objfun} yields
\[
\begin{split}
\delta&=\sum_{i\in S}\left(u_0^{n}-u_i^{n}+
\bm s_i^\top\bm\varphi
\right)^2 \\
&=\sum_{i\in S} \left((u_i^{n}-u_0^{n})^2-2(u_i^{n}-u_0^{n})
\bm s_i^\top\bm\varphi
+\bm\varphi^\top
\bm s_i\bm s_i^\top\bm\varphi
\right) \\
& = \bm\varphi^\top\mathbf G\bm\varphi+2\bm c^\top\bm\varphi+\text{const},
\end{split}
\]
where
\begin{equation}
\mathbf G=
\sum_{i\in S}\bm s_i\bm s_i^\top\quad\text{~and~}\quad 
\bm c=
-\sum_{i\in S}
(u_i^{n}-u_0^{n})\bm s_i.
\label{eq:coeff1}
\end{equation}

The constraints \eqref{eq:cons} can also be formulated in a compact form, namely
\begin{equation*}
\begin{split}
&m_{0j}^{n}\le v_{h,0}^{n}(\bm z_{jq}) =u_0^{n}+
\bm a(\bm z_{jq})^\top\bm\varphi\le M_{0j}^{n},\quad \qq,\jj,\\
&m_{00}^{n}\le v_{h,0}^{n}(\bx_0)=u_0^{n}+
\bm a(\bx_0)^\top\bm\varphi\le M_{00}^{n},
\end{split}
\end{equation*}
where $\bm a(\bx)$ is a vector-valued function defined by
\begin{equation}
  \bm a(\bx)=
  \begin{bmatrix}
    (\bx-\bx_0)/h \\
    \vech((\bx-\bx_0)\otimes(\bx-\bx_0) -\bJ_0)/h^2
  \end{bmatrix}.
  \label{eq:coeff2}
\end{equation}
Next we introduce the matrix notations
\begin{equation}
  \mathbf A=
  \begin{bmatrix}
    \bm a(\bx_0)^\top \\
    \bm a(\bm z_{11})^\top \\
    \vdots \\
    \bm a(\bm z_{1Q_1})^\top \\
    \vdots \\
    \bm a(\bm z_{J1})^\top \\
    \vdots \\
    \bm a(\bm z_{JQ_J})^\top 
  \end{bmatrix}, \quad
  \bm b = 
  \begin{bmatrix}
    m_{00}^n-u_0^n \\
    m_{01}^n-u_0^n \\
    \vdots \\
    m_{01}^n-u_0^n \\
    \vdots \\
    m_{0J}^n-u_0^n \\
    \vdots \\
    m_{0J}^n-u_0^n 
  \end{bmatrix}, \quad
  \bm B = 
  \begin{bmatrix}
    M_{00}^n-u_0^n \\
    M_{01}^n-u_0^n \\
    \vdots \\
    M_{01}^n-u_0^n \\
    \vdots \\
    M_{0J}^n-u_0^n \\
    \vdots \\
    M_{0J}^n-u_0^n 
  \end{bmatrix}. 
  \label{eq:AbB}
\end{equation}
Then the optimization problem above will be reduced to a  double-inequality
constrained \emph{quadratic programming problem} for variables $\bm\varphi$
\begin{equation}
\begin{split}
\min~&\dfrac{1}{2}\bm\varphi^\top\mathbf G\bm\varphi+\bm c^{\top}\bm\varphi\\
\st~ & \bm b\le \mathbf A\bm{\varphi}\le \bm B,
\end{split}
\label{eq:mainqp}
\end{equation}
where the coefficients are specified in \eqref{eq:bounds} -- \eqref{eq:AbB}.

The matrix $\mathbf G$ depends only on the geometry of the neighborhood of
$T_0$. For most of the cases the matrix $\mathbf G$ is positive-definite, and
hence the problem \eqref{eq:mainqp} becomes strictly convex. Moreover, the
feasible region is non-empty since the null solution $\bm\varphi=\bm 0$ is
always feasible. As a result, the global solution of problem \eqref{eq:mainqp}
exists uniquely, which we represent as
\begin{equation}
\bm{\varphi}=\mathcal Q(\mathbf G,\bm c,\mathbf A,\bm b,\bm B).
\label{eq:defQ}
\end{equation}
The operator $\mathcal{R}_h$ can thereby be defined through the reconstructed
quadratic polynomial as
\[
v_{h,0}^{n}(\bx)=\mathcal{R}_h[u_h^n, v_h^{n-1}](\bx) :=
u_0^{n}+\bm\varphi^\top\bm a(\bx), \quad \forall \bx \in T_0.
\]

\begin{remark}
  If we drop the constraints in the problem \eqref{eq:mainqp}, the resulting
  reconstruction becomes the $k$-exact reconstruction with $k=2$.  The
  corresponding solution is simply 
	\begin{equation}
  \bm{\varphi}_{\mathrm{ls}} =-\mathbf G^{-1}\bm c.
	\label{eq:phils}
	\end{equation}
\end{remark}

\begin{figure}[htbp]
	\centering
	\subfloat[Rectangular mesh]{\includegraphics[width=.22\textwidth]{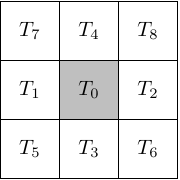}}
	\quad
	\subfloat[Triangular mesh I]{\includegraphics[width=.24\textwidth]{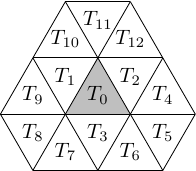}}
	\quad
	\subfloat[Triangular mesh II]{\includegraphics[width=.22\textwidth]{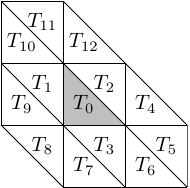}}
	\caption{Labels of Moore neighbors.}
	\label{fig:labeln}
\end{figure}

\begin{example}[rectangular meshes]
  Next we consider a square mesh with spacing $h$ (Fig. \ref{fig:labeln}(a)).
  The quadratic profile here takes the form
	\[
	\begin{split}
	v_{h,0}^n(x,y)&=u_0+L_x(x-x_0)+L_y(y-y_0)+
	\dfrac{1}{2}H_{xx}(x-x_0)^2+\dfrac{1}{2}H_{yy}(y-y_0)^2\\
	&+H_{xy}(x-x_0)(y-y_0)-\dfrac{1}{24}(H_{xx}+H_{yy})h^2.
	\end{split}
	\]			
  The optimal variables are
  $\bm{\varphi}=[hL_x,hL_y,h^2H_{xx}/2,h^2H_{xy},h^2H_{yy}/2]^\top$.  And the
  coefficients of the objective function are 
	\[
	\mathbf G=
	\begin{bmatrix}
	6 & & & &  \\
	& 6 & & &  \\
	& & 6 & &  4 \\
	& & & 4 & \\
	& & 4 &  & 6 
	\end{bmatrix}\quad\text{~and~}\quad
	\bm c=-
	\begin{bmatrix}
	(u_2^n-u_1^n)+(u_6^n-u_5^n)+(u_8^n-u_7^n)\\
	(u_4^n-u_3^n)+(u_7^n-u_5^n)+(u_8^n-u_6^n)\\
	u_1^n+u_2^n+u_5^n+u_6^n+u_7^n+u_8^n-6u_0^n\\
	(u_5^n+u_8^n)-(u_6^n+u_7^n)\\
	u_3^n+u_4^n+u_5^n+u_6^n+u_7^n+u_8^n-6u_0^n
	\end{bmatrix}.
	\]
\end{example}

\begin{example}[triangular meshes]
  Here we use two special cases to illustrate the triangular meshes: the
  equilateral triangular mesh (Fig. \ref{fig:labeln}(b)) and the diagonal
  triangular mesh (Fig. \ref{fig:labeln}(c)).  Choose $h$ to be the minimum side
  of the triangles.  A simple calculation yields the following expression of
  $\mathbf G$:
	\[
	\renewcommand{\arraystretch}{.75}   	 
	\dfrac{1}{24}
	\begin{bmatrix}
	132 & 0 & 0 & -14\sqrt{3} &  0 \\
	0 & 132 & -14\sqrt{3} & 0 & 14\sqrt{3} \\
	0 & -14\sqrt{3} & 105 & 0 & 35 \\
	-14\sqrt{3}& 0 & 0 & 35 & 0 \\
	0 & 14\sqrt{3} & 35 &  0 & 105 
	\end{bmatrix}
	\quad\text{~and~}\quad
	\renewcommand{\arraystretch}{1}   	 
	\dfrac{1}{9}
	\begin{bmatrix}
	66 & -33 & 14 & -7 &  -7 \\
	-33 & 66 & -7 & -7 & 14 \\
	14 & -7 & 70 & -35 & 35 \\
	-7 & -7 & -35 & 35 & -35 \\
	-7 & 14 & 35 &  -35 & 70 
	\end{bmatrix}.
	\]
  From this example we can expect that the matrix $\mathbf G$ is far from singularity
  for most triangular meshes. 
\end{example}

The main computation cost of the integrated quadratic reconstruction is the
solution of quadratic programming problems \eqref{eq:defQ}. An efficient and
robust quadratic programming solver becomes essential.  Here we use the
\emph{active-set method} \cite{Nocedal2006} to solve the problem
\eqref{eq:mainqp}. This method updates the solution by solving a series of
quadratic programming problems in which some of the inequalities are imposed as
equalities. We repeatedly estimate the active set until the solution reaches
optimality. To be more specific, let $\bm\varphi_k$ be solution of the $k$-th
iterative step, then the descending direction $\bm p_k$ and Lagrange
multipliers $\bm\lambda_k$ can be found by successively solving the following
two linear systems
\begin{align*}
(\mathbf M\mathbf G^{-1}\mathbf M^\top)\bm\lambda_k&=
\mathbf M(\bm\varphi_k+\mathbf G^{-1}\bm c),\\
\mathbf G\bm p_k&=\mathbf M^\top \bm\lambda-\mathbf G\bm\varphi_k-\bm c,
\end{align*} 
where the rows of matrix $\mathbf M$ are composed of normals of the active
constraints at the current step.  For a nonzero descending direction $\bm p_k$,
we set $\bm\varphi_{k+1}=\bm\varphi_k+\alpha_k\bm p_k$.  The step-length
parameter $\alpha_k$ is given by
\[
\alpha_k:=\min\left\{1,\min_{l}
\beta_l\right\},\quad 
\beta_l
=
\begin{cases}
\dfrac{b_l-\bm a_l^\top\bm{\varphi}_k}{\bm a_l^\top\bm p_k},
& \bm a_l^\top\bm p_k < 0, \\
\dfrac{B_l-\bm a_l^\top\bm{\varphi}_k}{\bm a_l^\top\bm p_k},
& \bm a_l^\top\bm p_k > 0, \\
+\infty, & \bm a_l^\top\bm p_k=0,
\end{cases}
\]
where $\bm a_l^\top$, $b_l$ and $B_l$ represent the $l$-th rows of matrices
$\mathbf A$, $\bm b$ and $\bm B$ respectively.  On the other hand, if $\bm
p_k=\bm 0$, then we check the signs of Lagrange multipliers.  We have achieved
the optimality if all the multipliers are non-negative; otherwise, we can find a
feasible direction by dropping the constraint with the most negative multiplier.
The initial guess is simply taken as the null solution, i.e.  $\bm\varphi_0=\bm
0$.  

With the operator $\mathcal{R}_h$ specified by the procedure above, the
numerical scheme \eqref{eq:firstorderscheme} is then closed, while it leads to
only first-order temporal accuracy.  To match the third-order spatial accuracy,
we adopt the SSP Runge-Kutta methods \cite{Gottlieb2001}, which is a multi-stage
combination of \eqref{eq:firstorderscheme}. The whole discretization scheme then
reads
\[
\begin{cases}
  u_h^{*}=u_h^{n}+\Delta t_n\mathcal L(v_h^{*}), &\quad 
  v_h^{*}=\mathcal{R}_h [u_h^{n},v_h^{n-1}],\\
  u_h^{**}=\dfrac{3}{4}u_h^{n}+\dfrac{1}{4}\left(u_h^{*}+\Delta
    t_n\mathcal L(v_h^{**})\right), &\quad 
  v_h^{**}=\mathcal{R}_h [u_h^{*},v_h^{*}],\\ 
  u_h^{n+1}=\dfrac{1}{3}u_h^{n}+\dfrac{2}{3}\left(u_h^{**} +\Delta
    t_n\mathcal L(v_h^{n})\right), &\quad 
  v_h^{n}= \mathcal{R}_h [u_h^{**},v_h^{**}].
\end{cases}
\]
And the initial value of the SSP Runge-Kutta method is still given by
\eqref{eq:initial-value}. 


\section{Accuracy and Stability}

In this section we study the accuracy and stability of the proposed scheme.
Roughly speaking, the third-order temporal accuracy is provided by the SSP
Runge-Kutta scheme already, thus we require a third-order spatial accuracy to
achieve an overall third-order accuracy in the truncation error. 

Basically, it can be shown that the quadratic reconstruction proposed above
provides us a third-order spatial accuracy for smooth functions.  To justify
this point, we need to study the asymptotic behavior of the quadratic
programming problem used to define the operator $\mathcal{R}_h$. In fact, let
$\mathcal P_h=\{T_{i}\}_{i\in S}$ be a family of cell patches, where the
relative position of the cells are the same, and hence $\bm r_{i}\propto
h$, $\bJ_{i}\propto h^2$, etc.  For the sake of convenience, the position of
centroid of $T_0$ is fixed.  To derive the continuous limit of problem
\eqref{eq:mainqp}, we first study the asymptotic expansions of the coefficients.
Obvious both $\mathbf G$ and $\mathbf A$ are scale-invariant.  Concerning the
other coefficients, we have the following lemma:
\begin{lemma}\label{lmm:1}
	There exist scale-invariant tensors 
	$\overline{\mathbf c}\in\mathbb R^{d\times d}$,
	$\overline{\overline{\mathbf c}}\in \mathbb R^{d\times d\times d}$,
	$\overline{\mathbf b},\overline{\mathbf B}\in
	\mathbb R^{(JQ+1)\times d}$ and
	$\overline{\overline{\mathbf b}},
	\overline{\overline{\mathbf B}}
	\in\mathbb R^{(JQ+1)\times d\times d}$ 
	such that the following asymptotic expansions hold
	\begin{subequations}
	\begin{align} 
	\bm c&=h\overline{\mathbf c}\cdot\nabla u(\bx_0)
	+h^2\overline{\overline{\mathbf c}}:\nabla \nabla u(\bx_0)
	+\mathcal O(h^3),\\
	\bm b&=h\overline{\mathbf b}\cdot\nabla u(\bx_0)
	+h^2\overline{\overline{\mathbf b}}:\nabla \nabla u(\bx_0)
	+\mathcal O(h^3),\\
	\bm B&=h\overline{\mathbf B}\cdot\nabla u(\bx_0)
	+h^2\overline{\overline{\mathbf B}}:\nabla \nabla u(\bx_0)
	+\mathcal O(h^3).
	\end{align}	\label{eq:asymexpan}
	\end{subequations}
\end{lemma}

\begin{proof}
  Here we investigate the asymptotic expansion (\ref{eq:asymexpan}a) of $\bm c$.
  The expansions of $\bm b$ and $\bm B$ can be analyzed in a similar manner.
  Note that the Taylor expansion of $u(\bx)$ about $\bx_0$ gives
  \[
  u(\bx)=u(\bx_0)+(\bx-\bx_0)^\top\nabla u(\bx_0) 
  +\dfrac{1}{2}(\bx-\bx_0)^\top \nabla\nabla u(\bx_0) (\bx-\bx_0)+\mathcal O(h^3).
  \]
  Therefore, the cell averages can be expressed as
  \begin{align*}
  &u_0=\aint_{T_0} u(\bx)\mathrm d\bx
  =u(\bx_0)+\dfrac{1}{2}\mathbf J_0:\nabla\nabla u(\bx_0)+\mathcal O(h^3),\\
  &u_i=\aint_{T_i} u(\bx)\mathrm d\bx
  =u(\bx_0)+\bm r_i\cdot\nabla u(\bx_0)+
  \dfrac{1}{2}(\bm r_i\otimes\bm r_i+\mathbf J_i):\nabla\nabla u(\bx_0)+\mathcal O(h^3),
  \end{align*}
  and as a result,
  \[
  u_i-u_0=
  \bm r_i\cdot\nabla u(\bx_0)
  +\dfrac{1}{2}(\bm r_i\otimes\bm r_i +\mathbf J_i-\mathbf J_0):\nabla\nabla u(\bx_0)+\mathcal O(h^3).
  \]	   
  The first-order coefficient then satisfies
  \[
  \begin{split}
  \bm c& = -\sum_{i\in S} 
  \bm s_i
  (u_i-u_0) \\
  & = - \sum_{i\in S} 
  \bm s_i
  \left( \bm r_i\cdot\nabla u(\bx_0)+
  \dfrac{1}{2}
  (\bm r_i\otimes\bm r_i+\bJ_i-\bJ_0):\nabla \nabla u(\bx_0)
  \right) +\mathcal O(h^3) \\
  & = 
  -\sum_{i\in S} 
  (\bm s_i\otimes \bm r_i)\cdot\nabla u(\bx_0)
  -\dfrac{1}{2}\sum_{i\in S} 
  (\bm s_i\otimes (\bm r_i\otimes\bm r_i+\bJ_i-\bJ_0)):
  \nabla\nabla u(\bx_0)+\mathcal O(h^3).
  \end{split}
  \]
  From here we can identify the expansion coefficients $\overline{\mathbf c}$
  and $\overline{\overline{\mathbf c}}$ in (\ref{eq:asymexpan}a).
\end{proof}

With the aid of expansions \eqref{eq:asymexpan}, we are able to derive the
continuous limit of the quadratic programming problem \eqref{eq:mainqp}.
Indeed, introduce a new variable $\bm{\psi}=\bm{\varphi}/h$, then the problem
\eqref{eq:mainqp} can be turned into the following equivalent form
\begin{equation*}
\begin{split}
\min~&\dfrac{1}{2}\bm\psi^\top \mathbf G
\bm\psi+h^{-1}\bm c^{\top}\bm\psi\\
\st~ & h^{-1}\bm b\le \mathbf A\bm\psi\le h^{-1}\bm B.
\end{split}
\end{equation*}
From here we can see that the continuous limit of the problem \eqref{eq:mainqp}
is
\begin{equation}
\begin{split}
\min~&\dfrac{1}{2}\bm\psi^\top\mathbf G
\bm\psi+(\overline{\mathbf c}\cdot\nabla u(\bx_0))^\top\bm\psi\\
\st~ & \overline{\mathbf b}\cdot\nabla u(\bx_0)\le 
\mathbf A\bm\psi\le \overline{\mathbf B}\cdot\nabla u(\bx_0).
\end{split}
\label{eq:conlimit}
\end{equation}
The above limiting problem provides us a precise statement of the well-posedness
of the problem \eqref{eq:mainqp}, which is a prerequisite of the accuracy result
stated below.
\begin{theorem}\label{thm:reconacc}
  Suppose that the solution operator
  $\mathcal Q(\mathbf G,\cdot, \mathbf A,\cdot,\cdot)$
  is Lipschitz continuous in the neighborhood of
  $(\overline{\mathbf c}\cdot\nabla u(\bx_0),~ \overline{\mathbf b}\cdot\nabla
  u(\bx_0),~ \overline{\mathbf B}\cdot\nabla u(\bx_0))$,
  then for any function
  $u \in C^3(\Omega)\bigcap L^1(\Omega)$, we have 
  \[
  \|\mathcal{R}_h[u]-u\|_{Z_0}=\mathcal O(h^3),
  \]
  where the semi-norm $\|\cdot\|_{Z_0}$ is defined by
  $\| f\|_{Z_0} =\max_{\bm z\in Z_0} |f(\bm z)|$.
\end{theorem} 

\begin{proof}
	Denote the second-order Taylor polynomial of $u$ by
	\[
	q(\bx)=u(\bx_0)+(\bx-\bx_0)^\top\nabla u(\bx_0)
	+\dfrac{1}{2}(\bx-\bx_0)^\top\nabla\nabla u(\bx_0)(\bx-\bx_0).
	\]
	Obviously
	\[
	\mathcal{R}_h[q]=q
	\quad\text{~and~}\quad  
	\|q-u\|_{Z_0}=\mathcal O(h^3).
	\]
  Then by the triangle inequality, we have
	\[
	\begin{split}
 	      & \|\mathcal{R}_h[u]-u\|_{Z_0} \\
 	 \le & \|q-u\|_{Z_0}+\|\mathcal{R}_h[q] - q\|_{Z_0}+
 	   \|\mathcal{R}_h[u]-\mathcal{R}_h[q]\|_{Z_0} \\
	\le &\mathcal O(h^3) + \mathcal O(h^3)
	 + \max_{\bm z\in Z_0}\|\bm a(\bm z)\|
	\cdot 
	\bigg \|h\mathcal Q 
    \left(
    {\mathbf G},
    h^{-1}\bm c,
    {\mathbf A},
    h^{-1}\bm b,
    h^{-1}\bm B
    \right)
    -h\mathcal Q
    \Big(
    {\mathbf G},
    \overline{\mathbf c}\cdot\nabla u(\bx_0)\\
    +&h\overline{\overline{\mathbf c}}:\nabla \nabla u(\bx_0),
    {\mathbf A},
    \overline{\mathbf b}\cdot\nabla u(\bx_0)
    +h\overline{\overline{\mathbf b}}:\nabla \nabla u(\bx_0),
    \overline{\mathbf B}\cdot\nabla u(\bx_0)
    +h\overline{\overline{\mathbf B}}:\nabla \nabla u(\bx_0)
    \Big) \bigg \| \\
    =&\mathcal O(h^3)+\mathcal O(h)\cdot \mathcal O(h^2)\\
    =&\mathcal O(h^3).
	\end{split}
	\]	
\end{proof}

Of course this estimation is only valid for smooth functions.  For conservation
laws, solutions are so seldom to be smooth that the stability of the numerical
scheme is of one's more concern.  High-order reconstruction may introduce
spurious oscillations near discontinuities.  One needs some stability criterion,
such as the local maximum principle, to rule out solutions with spurious
oscillations.  Here we show that the forward Euler scheme \eqref{eq:fvm-euler}
with our reconstruction satisfies a local maximum principle. The third-order SSP
discretization will thereby satisfy the local maximum principle due to the
convex combination.  

Before verifying the local maximum principle, let us investigate the
decomposition of the second moment tensor for an arbitrary control volume, as is
stated in the following lemma:

\begin{lemma}\label{lemma:sm}
  Let $K_j$ be the $d$-dimensional hyper-pyramid formed by the centroid $\bx_0$
  of $T_0$ and its facet $e_j~(\jj)$.  Introduce the coefficients
  $\alpha_j=|K_j|/|T_0|~(\jj)$. Then the following decomposition formula for the
  second moment tensor holds
  \[
  \aint_{T_0} \bx\otimes\bx \dd\bx 
  =\dfrac{d}{d+2}\sum_{j=1}^J\alpha_j \aint_{e_j} \bx\otimes\bx\dd\bx
  +\dfrac{2}{d+2}\bx_0\otimes\bx_0.
  \]
\end{lemma}

\begin{proof}
  At first, let us prove for any positive exponent $k$. One observes that
  \begin{equation}
    \aint_{T_0}(\bx-\bx_0)^{\otimes k} \dd\bx 
    =\sum_{j=1}^J\alpha_j 
    \aint_{K_j} (\bx-\bx_0)^{\otimes k}\dd\bx
    =\dfrac{d}{d+k} \sum_{j=1}^J \alpha_j
    \aint_{e_j} (\bx-\bx_0)^{\otimes k}\dd\bx,
    \label{iqr:kmoment}
  \end{equation}
  where $\bx^{\otimes k}$ denotes a tensor product of $\bx$ by $k$
  times to produce a $k$-th order tensor. Actually, the control volume
  $T_0$ is composed by a set of hyper-pyramids $K_1,\cdots,K_J$, thus
  we have
  \[
    \int_{T_0}(\bx-\bx_0)^{\otimes k} \dd\bx=\sum_{j=1}^J
    \int_{K_j}(\bx-\bx_0)^{\otimes k} \dd\bx.
  \]
  By the geometry of the hyper-pyramid $K_j$, we have
  \begin{equation}
    \int_{K_j} (\bx-\bx_0)^{\otimes k} \dd\bx=\int_0^1 l_j\dd\lambda \int_{\lambda e_j} (\bx-\bx_0)^{\otimes k} \dd \bx,
    \label{iqr:kmint}
  \end{equation}
  where $l_j$ is the height of $K_j$, and $\lambda e_j$ presents the
  cross section of hyper-pyramid $K_j$ parallel to the bottom surface
  $e_j$. The distance to the vertex is $\lambda l_j$, and $K_j$ is
  shown in the figure \ref{iqr:pyramid}.

  By a geometric similarity argument, we have
  \[
    \int_{\lambda e_j} (\bx-\bx_0)^{\otimes k} \dd
    \bx=\lambda^{d+k-1}\int_{e_j}(\bx-\bx_0)^{\otimes k} \dd \bx,
  \]
  and we insert it into \eqref{iqr:kmint} to have
  \[
    \int_{K_j} (\bx-\bx_0)^{\otimes k}
    \dd\bx=\dfrac{l_j}{d+k}\int_{e_j}(\bx-\bx_0)^{\otimes k} \dd \bx
    =\dfrac{d}{d+k}|K_j|\aint_{e_j}(\bx-\bx_0)^{\otimes k} \dd \bx.
  \]
  At last, we have that
  \[
    \int_{T_0}(\bx-\bx_0)^{\otimes k} \dd\bx=
    \dfrac{d}{d+k}\sum_{j=1}^J
    \alpha_j|T_0|\aint_{e_j}(\bx-\bx_0)^{\otimes k} \dd \bx.
  \]
  which prove \eqref{iqr:kmoment}.
  
  \begin{figure}[htbp]
    \centering
    \begin{tikzpicture}[scale=1]
      \coordinate (O) at (0,0);
      \coordinate (A) at (-1,-2.6);
      \coordinate (B) at (-0.5,-3.2);
      \coordinate (C) at (0.5,-3.2);
      \coordinate (D) at (1,-2.6);
      \coordinate (E) at (0,-2);
      \coordinate (A1) at (barycentric cs:O=.6,A=.4);
      \coordinate (B1) at (barycentric cs:O=.6,B=.4);
      \coordinate (C1) at (barycentric cs:O=.6,C=.4);
      \coordinate (D1) at (barycentric cs:O=.6,D=.4);
      \coordinate (E1) at (barycentric cs:O=.6,E=.4);
      \coordinate (annotation) at (barycentric cs:C=.5,D=.5);
      \coordinate (annotation1) at (barycentric cs:C1=.5,D1=.5);
      \fill [lightgray,opacity=.7] (A) -- (B) -- (C) -- (D) -- (E) -- cycle;
      \fill [lightgray,opacity=.7] (A1) -- (B1) -- (C1) -- (D1) -- (E1) -- cycle;
      \fill (O) circle (1pt) node[above]{$\bx_0$};
      \draw[thick] (O)--(A);
      \draw[thick] (O)--(B);
      \draw[thick] (O)--(C);
      \draw[thick] (O)--(D);
      \draw [densely dashed] (O)--(E);
      \draw[thick] (A)--(B)--(C)--(D);
      \draw[thick] (A1)--(B1)--(C1)--(D1);
      \draw [densely dashed] (A)--(E)--(D);
      \draw [densely dashed] (A1)--(E1)--(D1);
      \draw [-latex',thick] (annotation) + (.5,0) node[right] {$e_j$}-- (annotation);
      \draw [-latex',thick] (annotation1) + (.5,0) node[right] {$\lambda e_j$}-- (annotation1);
      \draw [|<->|,very thick] (1.8,-2.6) -- (1.8,0) node [midway,right] {$l_j$};
    \end{tikzpicture}
    \caption{hyper-pyramid $K_j$}\label{iqr:pyramid}
  \end{figure}

  In particular, we have
  \begin{align*}
  \aint_{T_0}(\bx-\bx_0)\otimes(\bx-\bx_0) \dd\bx 
  &=\dfrac{d}{d+2} \sum_{j=1}^J \alpha_j
  \aint_{e_j} (\bx-\bx_0)\otimes(\bx-\bx_0)\dd\bx,\\
  \bm 0=\aint_{T_0}(\bx-\bx_0)\dd\bx 
  &=\dfrac{d}{d+1} \sum_{j=1}^J \alpha_j
  \aint_{e_j} (\bx-\bx_0)\dd\bx.
  \end{align*}
  And as a result
  \begin{align*}
  \aint_{T_0}(\bx-\bx_0)\otimes(\bx-\bx_0) \dd\bx 
  &=\dfrac{d}{d+2} \sum_{j=1}^J \alpha_j
  \aint_{e_j} (\bx-\bx_0)\otimes(\bx-\bx_0)\dd\bx,\\
  \aint_{T_0}(\bx-\bx_0)\otimes \bx_0 \dd\bx 
  &=\dfrac{d}{d+2} \sum_{j=1}^J \alpha_j
  \aint_{e_j} (\bx-\bx_0)\otimes \bx_0 \dd\bx,\\
  \aint_{T_0} \bx_0\otimes  (\bx-\bx_0)\dd\bx 
  &=\dfrac{d}{d+2} \sum_{j=1}^J \alpha_j
  \aint_{e_j} \bx_0 \otimes (\bx-\bx_0)\dd\bx,\\    
    \aint_{T_0} \bx_0\otimes \bx_0\dd\bx
  &=\dfrac{d}{d+2} \sum_{j=1}^J \alpha_j
  \aint_{e_j}\bx_0\otimes \bx_0\dd\bx
  +\dfrac{2}{d+2} \bx_0\otimes \bx_0.
  \end{align*}
  Summing up the above four identities yields the desired formula.
\end{proof}
Next we can establish the following quadrature rule
\begin{lemma}\label{lemma:qf}
  The following quadrature formula is exact for any quadratic polynomial $v$
  \begin{equation}
  \aint_{T_0} v(\bx)\dd\bx
  =\dfrac{d}{d+2}\sum_{j=1}^J\sum_{q=1}^{Q_j} \alpha_jw_{jq}v(\bm z_{jq})
  +\dfrac{2}{d+2}v(\bx_0).
  \label{eq:quadexpan}
  \end{equation}
\end{lemma}

\begin{proof}
   Since the quadrature rules specified on the facet $e_j$ is of at least second-order 
   accuracy, we have
  \[
  \sum_{q=1}^{Q_j} w_{jq}\bm z_{jq}=\aint_{e_j} \bx\dd\bx\quad\text{~and~}\quad
  \sum_{q=1}^{Q_j} w_{jq}\bm z_{jq}
  \otimes \bm z_{jq}
  =\aint_{e_j} \bx\otimes \bx\dd\bx.
  \]
  Using the formula in Lemma \ref{lemma:sm} we know that
  \begin{align*}
  &\dfrac{d}{d+2}\sum_{j=1}^J\sum_{q=1}^{Q_j}\alpha_jw_{jq}\bm z_{jq}
  +\dfrac{2}{d+2}\bx_0
  =\bx_0
  =\aint_{T_0}\bx\dd\bx, \\
  &\dfrac{d}{d+2}\sum_{j=1}^J\sum_{q=1}^{Q_j}\alpha_jw_{jq}\bm z_{jq}\otimes \bm z_{jq}
  +\dfrac{2}{d+2}\bx_0\otimes \bx_0
  =\bJ_0+\bx_0\otimes \bx_0
  =\aint_{T_0}\bx\otimes \bx\dd\bx.
  \end{align*}
  From this we conclude that \eqref{eq:quadexpan} holds for the function
  $v:\bx\rightarrow \bx$ and $v:\bx\rightarrow \bx\otimes\bx$, and so is any 
  quadratic function $v$. This completes the proof.
\end{proof}

With the aid of the formula \eqref{eq:quadexpan}, we are able to verify a local maximum principle
for the finite volume scheme \eqref{eq:fvm-euler} following the line in \cite{Zhang2012}.

\begin{theorem}\label{thm:lmp}
  Suppose that $\mathcal T$ is a convex $d$-polytope grid.  Let $L_{\min}$ be
  the minimum distance from the centroid of any given cell to all the facets of
  this cell. Moreover, let $\mathcal F$ be a monotone $C^1$ numerical flux
  function. Then the finite volume scheme \eqref{eq:fvm-euler} with
  integrated quadratic reconstruction fulfills the following local maximum
  principle
  \begin{equation}
    \min_{0\le j\le J}\min_{\bm z\in Z_j}{v_{h,j}^{n}(\bm z)}
    \le u_0^{n+1}\le 
    \max_{0\le j\le J}\max_{\bm z\in Z_j}{v_{h,j}^{n}(\bm z)},
    \label{eq:genmp}
  \end{equation}
  under the CFL condition
  \begin{equation}
  \Delta t_n\sup_{u^-,u^+,\bm n}\dfrac{\partial\mathcal F(u^-,u^+;\bm n)}{\partial u^-}
   \le\dfrac{L_{\min}}{d+2}.
   \label{eq:cflcond}
  \end{equation}
  \end{theorem}
\begin{proof}
Inserting the quadrature formula \eqref{eq:quadexpan} into the finite volume
scheme \eqref{eq:fvm-euler} yields
\[
u_0^{n+1}=
\dfrac{2}{d+2}v_{h,0}^n(\bx_0)+
\sum_{j=1}^J\sum_{q=1}^{Q_j}\left(
\dfrac{d\alpha_j}{d+2}w_{jq}v_{h,0}^n(\bm z_{jq})-\dfrac{\Delta t_n}{|T_0|}w_{jq}
\mathcal F(v_{h,0}^n(\bm z_{jq}),v_{h,j}^n(\bm z_{jq});\bm n_j)|e_j|\right),
\]
If we take the right-hand side of the above scheme as a function 
\[
u_0^{n+1}=\mathcal H(v_{h,0}^{n}(\bm z_{11}),\cdots,v_{h,0}^{n}
(\bm z_{JQ_J}),v_{h,1}^{n}(\bm z_{11}),\cdots,v_{h,J}^{n}(\bm z_{JQ_J})
,v_{h,0}^{n}(\bx_0)),
\]
we then have that
\begin{align*}
\dfrac{\partial \mathcal H}{\partial v_{h,0}^n(\bm z_{jq})}&
=\dfrac{d\alpha_jw_{jq}}{d+2}-
\dfrac{\Delta t_n|e_j|w_{jq}}{|T_0|}\cdot
\dfrac{\partial\mathcal F(v_{h,0}^n(\bm z_{jq}),v_{h,j}^n(\bm z_{jq});\bm n_j)}
{\partial v_{h,0}^n(\bm z_{jq})}\\
&\ge\dfrac{d\alpha_jw_{jq}}{d+2}-
\dfrac{\Delta t_n|e_j|w_{jq}}{|T_0|}
\sup_{u^-,u^+,\bm n} \dfrac{\partial\mathcal F(u^-,u^+;\bm n)}{\partial u^-},\\
\dfrac{\partial \mathcal H}{\partial v_{h,j}^n(\bm z_{jq})}&
=-\dfrac{\Delta t_n|e_j|w_{jq}}{|T_0|}\cdot\dfrac{\partial
\mathcal F(v_{h,0}^n(\bm z_{jq}),v_{h,j}^n(\bm z_{jq});\bm n_j)}{\partial v_{h,j}^n(\bm z_{jq})}\ge 0,\\
\dfrac{\partial \mathcal H}{\partial v_{h,0}^n(\bx_0)}&=\dfrac{2}{d+2}>0.
\end{align*}
As a result, $\mathcal H$ is non-decreasing with respect to each argument provided that 
\begin{equation*}
\Delta t_n\sup_{u^-,u^+,\bm n} \dfrac{\partial\mathcal F(u^-,u^+;\bm n)}{\partial u^-}
\le \dfrac{d}{d+2}\min_{1\le j\le J}\dfrac{\alpha_j|T_0|}{|e_j|}
=\dfrac{1}{d+2}\min_{1\le j\le J}\dfrac{d|K_j|}{|e_j|}.
\end{equation*}
Note that $d|K_j|/|e_j|$ is exactly the distance from the point $\bm x_0$ to the
facet $e_j$.  Therefore, a sufficient condition of the time restriction is
\[
\Delta t_n\sup_{u^-,u^+,\bm n} \dfrac{\partial\mathcal F(u^-,u^+;\bm n)}{\partial u^-}
\le\dfrac{L_{\min}}{d+2}.
\]
Also, we have $\mathcal H(u,\cdots,u)=u$ due to the consistency of numerical
flux functions. Denote
\[
u^{\min}=\min_{0\le j\le J}\min_{\bm z\in Z_j}{v_{h,j}^{n}(\bm z)}\quad\text{~and~}\quad
u^{\max}=\max_{0\le j\le J}\max_{\bm z\in Z_j}{v_{h,j}^{n}(\bm z)},
\]
then the monotonicity of $\mathcal H$ implies the desired local maximum principle 
\[
u^{\min}\le \mathcal H(u^{\min},\cdots,u^{\min})\le
u_0^{n+1}\le \mathcal H(u^{\max},\cdots,u^{\max})=u^{\max}.
\]	
\end{proof}
\begin{remark}
   Another form of CFL condition in terms of the mesh size $h$ is also useful
   \[
   a\Delta t_n\le \nu h,
   \]
   where $\nu$ is a CFL number.  Here we measure the mesh size $h$ of simplicial
   control volume (triangle or tetrahedron) by the diameter of its inscribed
   ball, whereas that of Cartesian control volume (rectangle or
   cuboid) the harmonic mean of its dimensions.  The value of  CFL number $\nu$
   under such definition is also listed in Table \ref{tab:sms}.
\end{remark}

Although the local maximum principle given in Theorem \ref{thm:lmp} is not
recursively formulated, we can verify the bound-preserving property, saying the
numerical solution at any time level is bounded by the initial solution.  More
specifically, we have
\begin{corollary}
  The finite volume scheme \eqref{eq:fvm-euler} with integrated quadratic
  reconstruction is bound-preserving, namely
  \begin{equation}
    \inf_{\bx \in\Omega}{u(\bx,0)}\le u_0^{n}\le \sup_{\bx \in\Omega}{u(\bx,0)},
    \label{eq:bpp}
  \end{equation}
  provided that the solution is advanced with a time step subjected to the CFL
  condition \eqref{eq:cflcond}.
\end{corollary}
\begin{proof}
	Introduce the following notations of global upper bounds at $n$-th time level
	\[
	M^{n}=\max_{T_0\in\mathcal T}{u_0^{n}},\quad
	\hat M^{n}=\max_{T_0\in\mathcal T}
	\max_{\bm z\in Z_0}{v_{h,0}^{n}(\bm z)},\quad
	n=0,1,2,\cdots.
	\]
  Since $u_0^n$ is a convex combination of $\{v_{h,0}^n(\bm z)\}_{\bm z\in Z_0}$, we
  have $u_0^{n}\le \max\limits_{\bm z\in Z_0}{v_{h,0}^{n}(\bm z)}$.  Taking the
  maximum over all $T_0\in \mathcal T$ yields the relation $M^{n}\le \hat
  M^{n}$.  For any cell $T_0\in\mathcal T$ and $\bm z\in Z_0$, the construction
  of integrated quadratic reconstruction yields
	\[
	v_{h,0}^{n}(\bm z)\le \max\limits_{0\le j\le J}{M_{0j}^{n}}
	\le \max\{\hat M^{n-1},M^{n}\},\quad\forall \bm z\in Z_0,
	\]
  and hence $\hat M^{n}\le \max\{\hat M^{n-1},M^{n}\}$.  Note that the local
  maximum principle \eqref{eq:genmp} implies that $M^{n}\le \hat M^{n-1}$.
  Therefore we have the monotonicity 
	\[
	\hat M^{n}\le \hat M^{n-1},\quad n=1,2,\cdots.
	\]
	On the other hand, the construction of initial time level yields
	\[
	v_{h,0}^{0}(\bm z)\le \max\limits_{0\le j\le J}{M_{0j}^{0}}
	\le \max\left\{ \max_{T_0\in\mathcal T}\max_{\bm z\in Z_0} u(\bm z,0),M^{0}\right\}
	\le \sup_{\bx \in\Omega}{u(\bx,0)},\quad\forall \bm z\in Z_0,
	\]
  and hence $\hat M^{0}\le\sup\limits_{\bx \in\Omega}{u(\bx,0)}$.  Finally we
  conclude that
	\[
	u_0^{n}\le M^{n}\le \hat M^{n}\le \cdots \le \hat M^{0}\le 
	\sup\limits_{\bx \in\Omega}{u(\bx,0)}.
	\]
	Similarly we can verify the left-hand side of the inequality \eqref{eq:bpp}.
\end{proof}    


\section{Numerical Results}

In this section we provide some numerical results to demonstrate the performance
of the integrated quadratic reconstruction.  The time step length is indicated
by the CFL number listed in Table \ref{tab:sms} if not otherwise specified. The
results are compared against the scaling limiter of Zhang \textit{et al.}
\cite{Zhang2012} acting on the 2-exact reconstruction, which we refer to as 
the scaling quadratic reconstruction (SQR) in our context.

\subsection{Two-dimensional linear equation}

This is a two-dimensional problem used to assess the order of accuracy.  We
solve the following linear equation
\[
u_t+u_x+2u_y=0,
\]
with initial profile given by the double sine wave function 
\[
u(x,y,0)=\sin(2\pi x)\sin(2\pi y).
\]
This problem has also been considered in
\cite{Hubbard1999,Park2010,May2013,Chen2016}.  The computational domain is
$[0,1]\times[0,1]$.  Periodic boundary conditions are applied.  We perform the
convergence test on both rectangular and triangular meshes.  The rectangular
mesh is uniform.  In the triangular mesh test, both structured and unstructured
meshes are examined.  The structured mesh is generated by dividing each
rectangular element along the diagonal direction, while the unstructured mesh is
generated by Delaunay triangulation.  In Table \ref{tab:2Dadv}, one observes
third-order of accuracy on various meshes.

\begin{figure}[!ht]
  \centering
  \captionof{table}{Accuracy for 2D linear equation.}
  \label{tab:2Dadv}
  \raisebox{-.1\textheight}
  {\includegraphics[width=.32\textwidth]{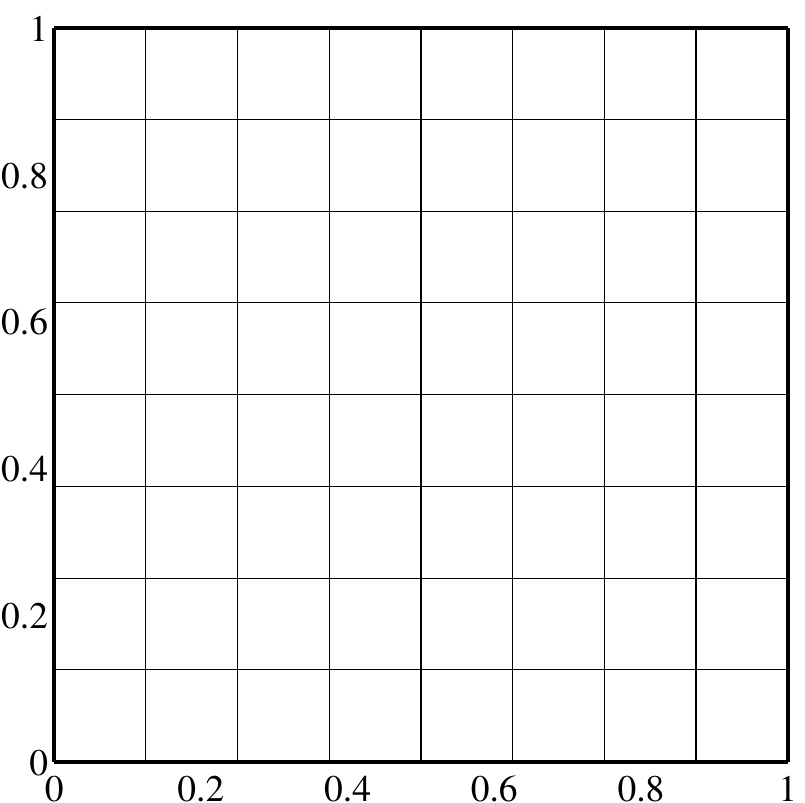}}
  \qquad
  \begin{tabular}{lrrrr}
    \toprule
    \multicolumn{5}{c}{\bf Rectangular meshes} \\
    \cmidrule(lr){1-5}
        $h$ & $L^1$ error & Order & 
    $ L^\infty$ error & Order \\ \midrule
    $1/8$ & 3.84E-01 & --- & 9.52E-01 & --- \\ 
    $1/16$  & 1.36E-01 & 1.50 & 3.38E-01 & 1.50 \\ 
    $1/32$ & 2.08E-02 & 2.71 & 5.34E-02 & 2.66 \\ 
    $1/64$ & 2.68E-03 & 2.96 & 7.44E-03 & 2.84 \\ 
    $1/128$ & 3.37E-04 & 3.00 & 1.09E-03 & 2.77 \\ 
    $1/256$ & 4.21E-05 & 3.00 & 1.79E-04 & 2.60 \\ 
    \bottomrule
  \end{tabular}	
  \raisebox{-.1\textheight}
  {\includegraphics[width=.32\textwidth]{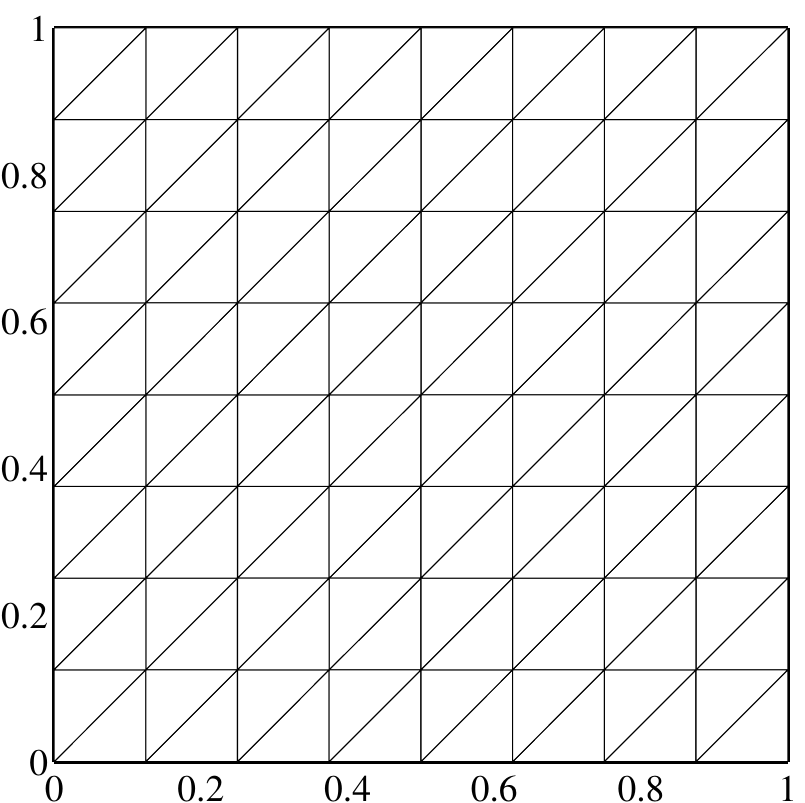}}
  \qquad
  \begin{tabular}{lrrrr}
    \toprule
    \multicolumn{5}{c}{\bf Structured triangular meshes} \\
    \cmidrule(lr){1-5}
    $h$ & $ L^1$ error & Order & 
    $ L^\infty$ error & Order \\ \midrule		
    $1 / 8$ & 2.79E-01 & --- & 5.14E-01 & ---  \\ 
    $1/ 16$ & 7.22E-02 & 1.95 & 1.27E-01 & 2.02 \\ 
    $1 / 32$ & 1.02E-02 & 2.82 & 1.82E-02 & 2.80 \\ 
    $1 / 64$ & 1.30E-03 & 2.97 & 2.54E-03 & 2.84  \\ 
    $1/ 128$ & 1.63E-04 & 3.00 & 3.75E-04 & 2.76 \\ 
    $1/ 256$ & 2.04E-05 & 3.00 & 6.22E-05 & 2.59  \\
    \bottomrule  	
  \end{tabular}	
  \raisebox{-.1\textheight}
  {\includegraphics[width=.32\textwidth]{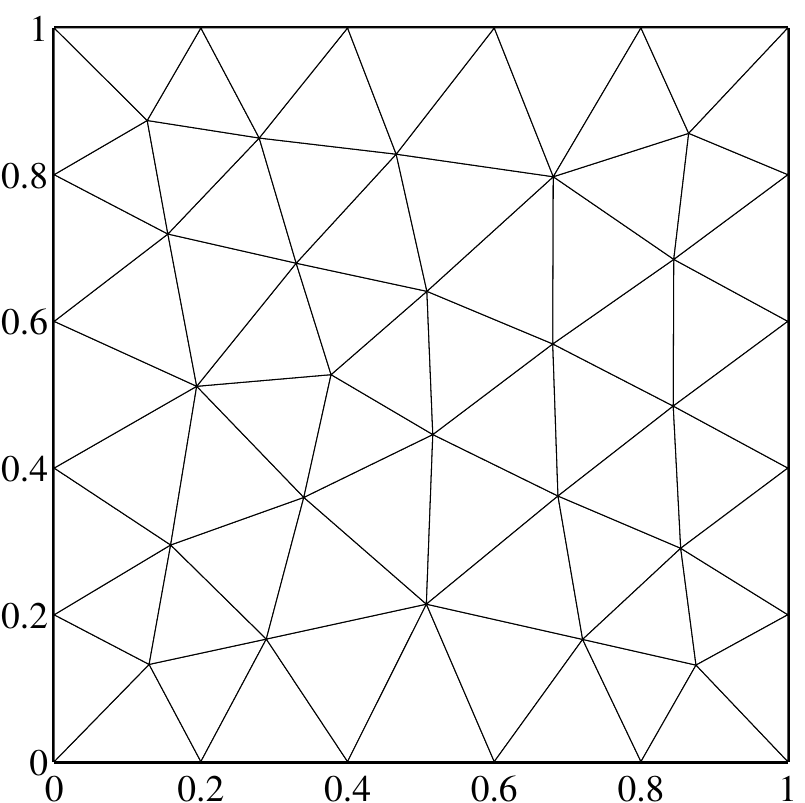}}
  \qquad
  \begin{tabular}{lrrrr}
    \toprule
    \multicolumn{5}{c}{\bf Unstructured triangular meshes} \\
    \cmidrule(lr){1-5}
    $h$ & $ L^1$ error & Order & 
    $ L^\infty$ error & Order \\ \midrule		
    $1 / 5$ & 3.40E-01 & --- & 8.24E-01 & --- \\ 
    $1 / 10$ & 1.05E-01 & 1.69 & 2.41E-01 & 1.78 \\ 
    $1 / 20$ & 1.58E-02 & 2.73 & 3.85E-02 & 2.64 \\ 
    $1 / 40$ & 2.05E-03 & 2.95 & 5.63E-03 & 2.77 \\ 
    $1 / 80$ & 2.59E-04 & 2.99 & 9.84E-04 & 2.52 \\ 
    $1/160$ & 3.25E-05 & 2.99 & 2.00E-04 & 2.30 \\
    \bottomrule  
  \end{tabular}	
\end{figure}

\subsection{Composite model problem}

In this example we use a composite model problem to test the robustness of the
proposed scheme.  The following linear advection equation is computed
\[
u_t+u_x+0u_y=0,
\]
on the square domain $[0,1]\times [0,1]$.  The initial profile consists of two
bulks with large discrepancy in magnitude, i.e.
\[
u(x,y,0) =
\begin{cases}
A, &  1/8 \le x \le 3/8 \text{~and~} 3/8\le y\le 5/8,\\
100, & 5/8 \le x \le 7/8 \text{~and~} 3/8\le y\le 5/8,\\
0, & \text{Otherwise}.
\end{cases}
\]
where $A=1$ or $10$ denotes the magnitude of the short bulk.

The solution of IQR and SQR schemes on a fine structured mesh at $t=0.1$ are
shown in Fig. \ref{fig:compositemodel}. Only the short bulk is displayed here
for the sake of visibility. We would like to mention that the global range
$[0,100]$ of the initial solution is required as a parameter in the SQR
scheme. It is observed that the SQR scheme produces severe oscillation on the
front edge of the short bulk, though the global range of the solution is
strictly preserved within the interval $[0,100]$. To quantitatively study the
behavior of the short bulk when the mesh is refined, we measure the peak value
of the short bulk on successively refined meshes, as is shown in Table
\ref{tab:peakv}. It is observed that the amount of the overshoot is about $11\%$
the magnitude of the short bulk for the SQR scheme, regardless of the resolution
of the numerical scheme. Indeed, due to the presence of the tall bulk, the
scaling limiter makes no difference to the edge of the short bulk.  On the other
hand, the IQR scheme maintains the range of the short bulk due to its locality
and parameter-free performance. From this perspective we confirm the robustness
of the IQR scheme.

\begin{figure}[htbp]
    \centering
    \begin{minipage}[b]{.25\textwidth}
    		\tdplotsetmaincoords{75}{-148}
    		\begin{tikzpicture}[tdplot_main_coords,scale=3]
    		\def\SHORT{0.012}
    		\def\TALL{1.2}
    		\filldraw[draw =  gray,fill = gray!70] (0,0,0)--(0,1,0)--(1,1,0)--(1,0,0)--cycle;
    		\coordinate (A1) at (1/8,3/8,0);
    		\coordinate (B1) at (1/8,5/8,0);
    		\coordinate (C1) at (3/8,5/8,0);
    		\coordinate (D1) at (3/8,3/8,0);
    		\coordinate (E1) at (1/8,3/8,\SHORT);
    		\coordinate (F1) at (1/8,5/8,\SHORT);
    		\coordinate (G1) at (3/8,5/8,\SHORT);
    		\coordinate (H1) at (3/8,3/8,\SHORT);
    		\filldraw [draw =gray,fill = gray!70] (E1)--(F1)--(G1)--(H1)--cycle;
    		\filldraw [draw =gray,fill = gray!50] (A1)--(B1)--(F1)--(E1)--cycle;
    		\filldraw [draw =gray,fill = gray!40] (B1)--(C1)--(G1)--(F1)--cycle;
    		\coordinate (A2) at (5/8,3/8,0);
    		\coordinate (B2) at (5/8,5/8,0);
    		\coordinate (C2) at (7/8,5/8,0);
    		\coordinate (D2) at (7/8,3/8,0);
    		\coordinate (E2) at (5/8,3/8,\TALL);
    		\coordinate (F2) at (5/8,5/8,\TALL);
    		\coordinate (G2) at (7/8,5/8,\TALL);
    		\coordinate (H2) at (7/8,3/8,\TALL);
    		\filldraw [draw =gray,fill = gray!70] (E2)--(F2)--(G2)--(H2)--cycle;
    		\filldraw [draw =gray,fill = gray!50] (A2)--(B2)--(F2)--(E2)--cycle;
    		\filldraw [draw =gray,fill = gray!40] (B2)--(C2)--(G2)--(F2)--cycle;
    		\fill [fill=white] (0.3,0.8)--++(0,-0.1)--++(0.2,0)--++(0,-0.07)--++(0.12,0.12)--++(-0.12,0.12)
    		--++(0,-0.07)--++(-0.2,0)--cycle;
    		\end{tikzpicture}
    		\caption{Compositor Model Problem}\label{fig:compositedia}
    \end{minipage}
    \quad
    \begin{minipage}[b]{.7\textwidth}
    \centering
    \subfloat[IQR]{\includegraphics[width=.45\textwidth]{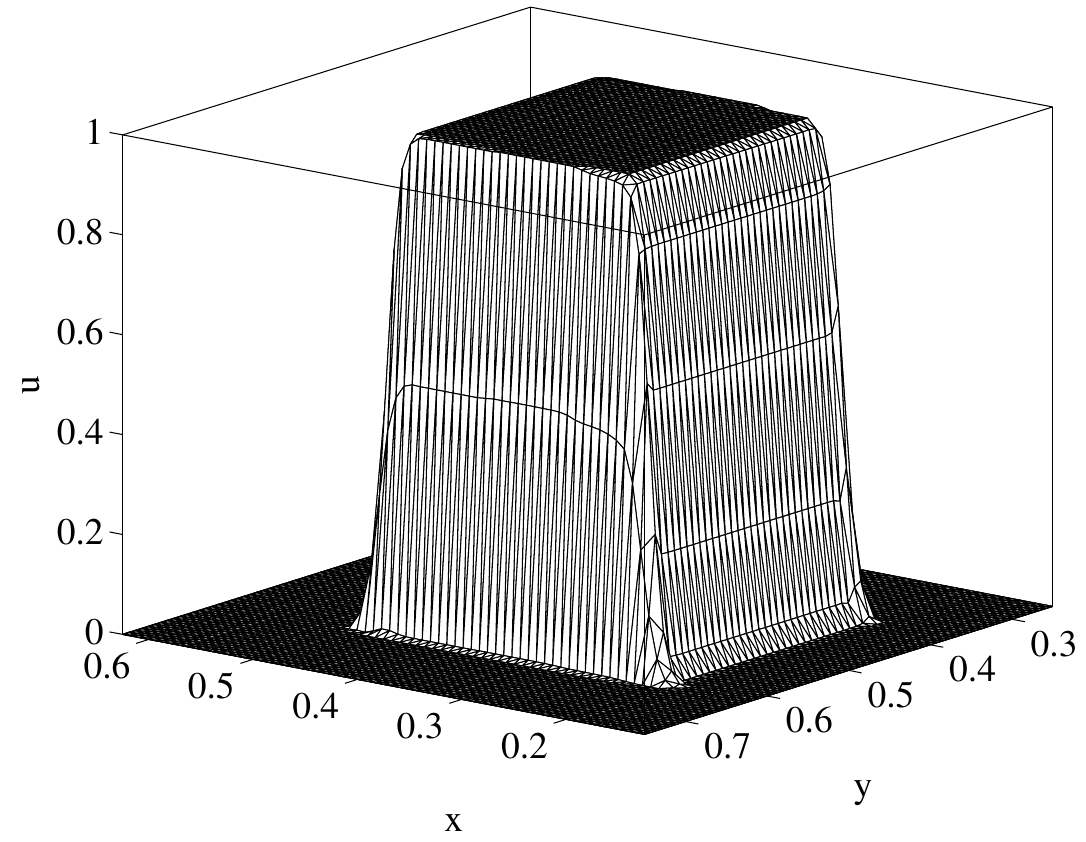}}
    \subfloat[SQR]{\includegraphics[width=.45\textwidth]{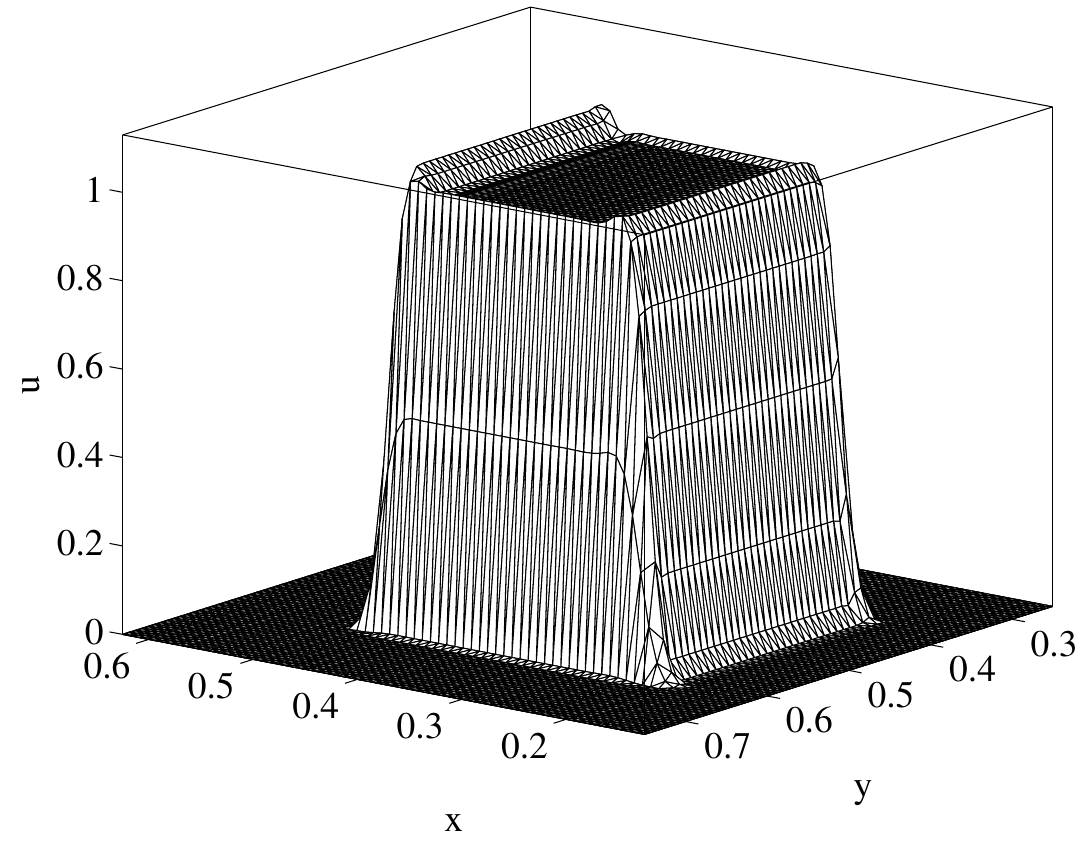}}
    \caption{Shape of the short bulk at $t=0.1$ with $128\times 128\times 2$ cells.}
    \label{fig:compositemodel}
    \end{minipage}
\end{figure}

\begin{table}[htbp]
	\centering
	\caption{Peak value of the short bulk at $t=0.1$ on successively refined meshes.}
  \label{tab:peakv}
\begin{tabular}{lrrrrr}
	\toprule
	& \multicolumn{2}{c}{$A=1$} &
	& \multicolumn{2}{c}{$A=10$} \\
    \cmidrule(lr){2-3} \cmidrule(lr){5-6}
	$h$   &  \bf SQR   &  \bf IQR   &&  \bf SQR  &  \bf IQR  \\
    \midrule 
	$1/16$ &  1.1441  &  1.0000  &&  11.4405 & 10.0000 \\
	$1/32$ &  1.1535  &  1.0000  &&  11.5353 & 10.0000 \\
	$1/64$ &  1.1268  &  1.0000  &&  11.2681 & 10.0000 \\
	$1/128$ & 1.1298  &  1.0000  &&  11.2984 & 10.0000 \\
	\bottomrule
\end{tabular}
\end{table}

\subsection{Solid body rotation problem}

This is a non-uniform scalar flow where the initial profile consists of smooth
hump, cone and slotted cylinder.  See \cite{LeVeque1996} for the algebraic
descriptions of the geometric shapes.  We solve the circular advection equation
\[
u_t-(y-0.5)u_x+(x-0.5)u_y=0,
\]
on $[0,1]\times[0,1]$ with homogeneous boundary conditions.  Fig. \ref{fig:sbr}
shows the results after one revolution on two levels of Delaunay meshes.  In the
finer mesh, the IQR scheme almost keeps the shape of the initial solution
without much distortion.

\begin{figure}[htbp]
	\subfloat[]{\includegraphics[width=.5\textwidth]
		{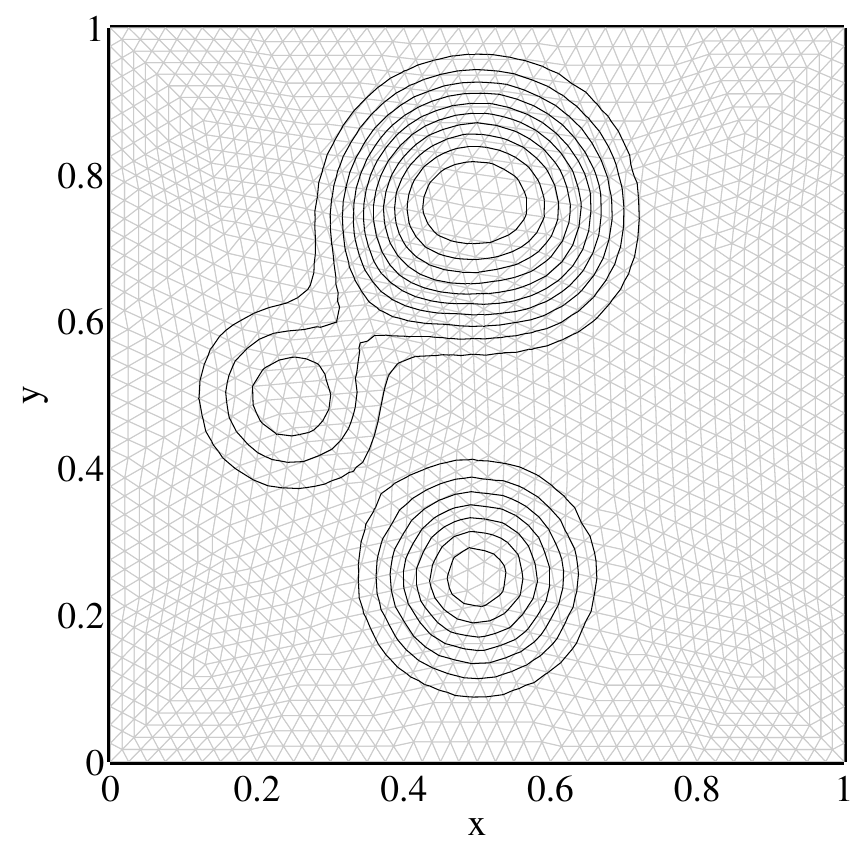}}
	\subfloat[]{\includegraphics[width=.5\textwidth]
		{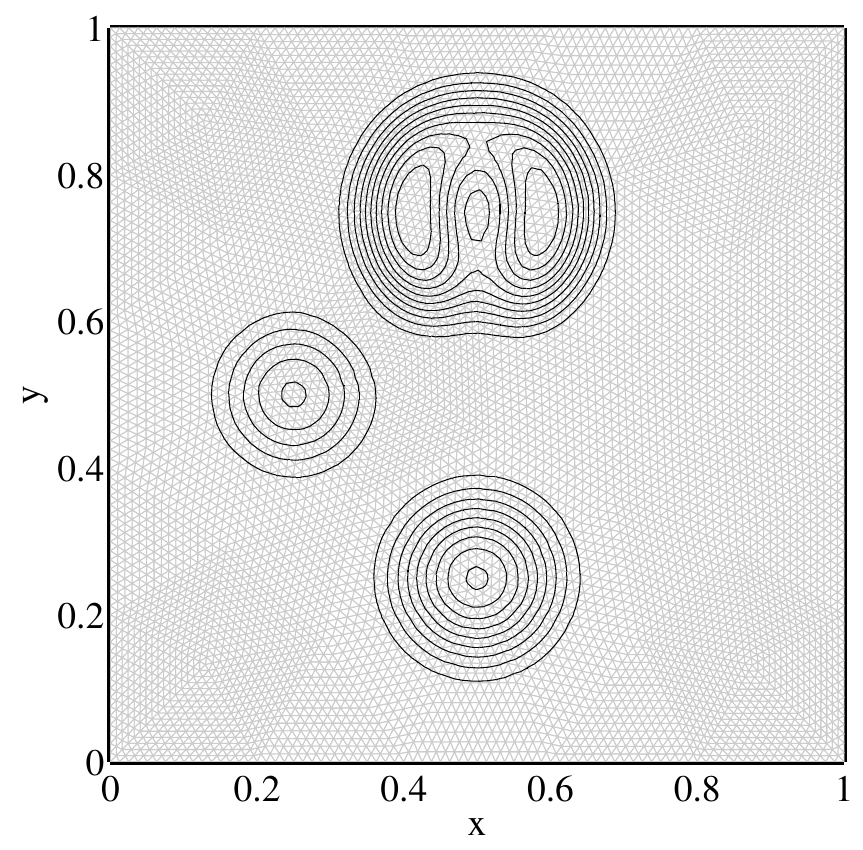}}\\
	\subfloat[]{\includegraphics[width=.5\textwidth]
		{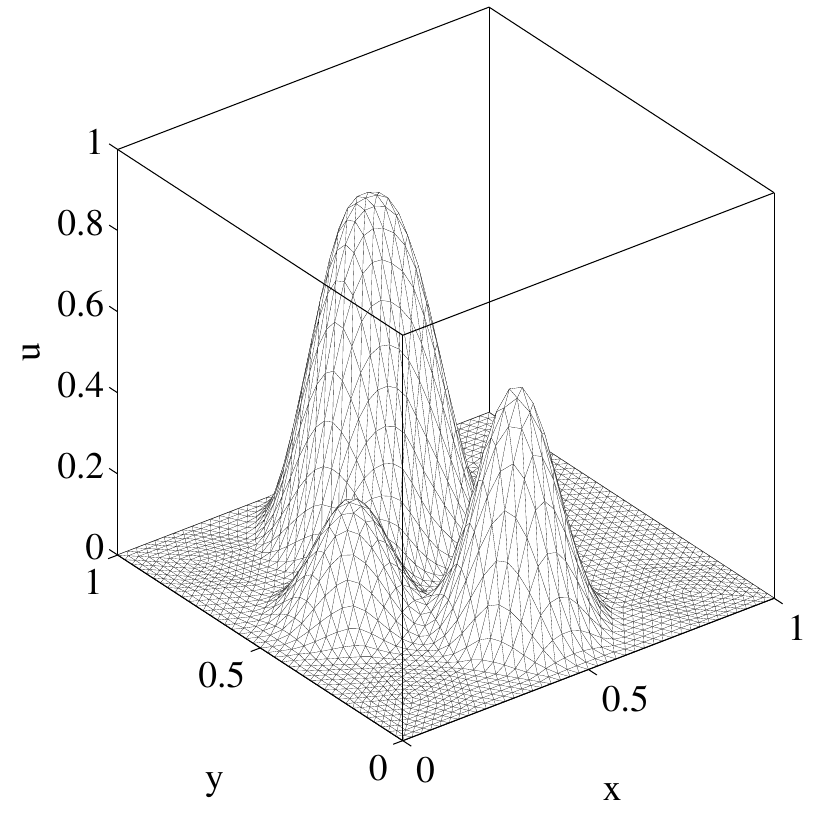}}	
	\subfloat[]{\includegraphics[width=.5\textwidth]
		{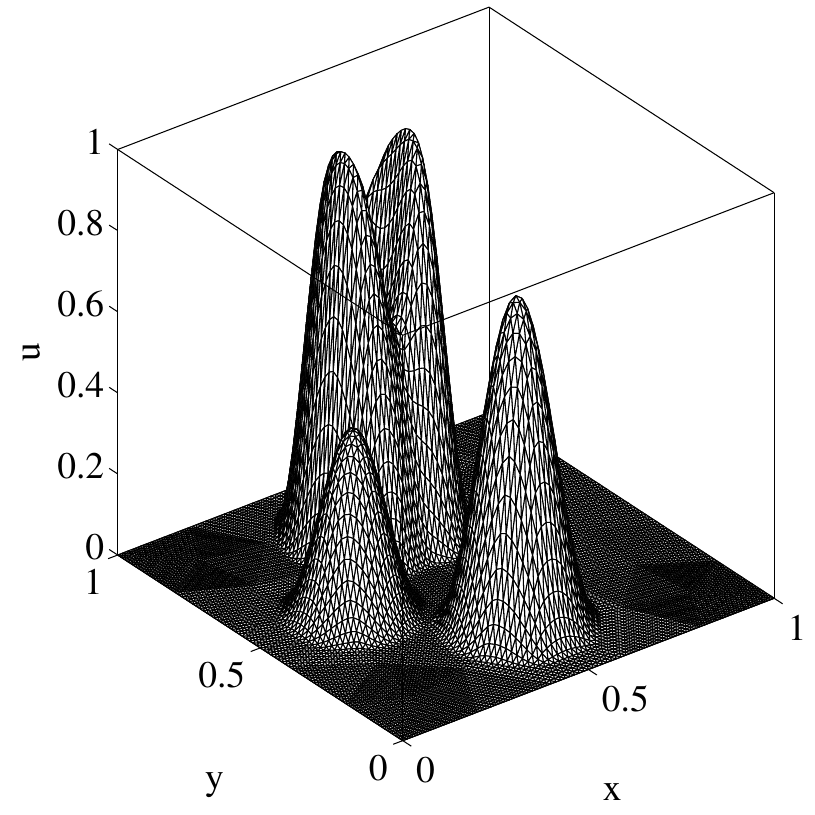}}	
	\caption{Solutions of the solid body rotation problem at $t=2\pi$: left panel 
		(4096 cells) and right panel (16384 cells).}
	\label{fig:sbr}
\end{figure}

\subsection{Two-dimensional Burgers' equation}

Following \cite{Hu1999}, we consider the two-dimensional Burgers' equation 
\begin{equation}
u_t+\left(\dfrac{1}{2}u^2\right)_x+\left(\dfrac{1}{2}u^2\right)_y=0,
\label{eq:burgers}
\end{equation}
with initial condition $u(x,y,0)=0.3+0.7\sin({\pi}(x+y)/2)$ on the domain
$[-2,2]\times [-2,2]$.  Periodic boundary conditions are applied.  To assess the
order of accuracy on smooth regions we advance the solution until $t=0.5/\pi^2$.
The exact solution at a given position $(x,y)$ can be found by applying a
fixed-point iteration to the nonlinear algebraic equation
\[
u=0.3+0.7\sin\left(\dfrac{\pi}{2}(x+y)-\dfrac{u}{2\pi}\right).
\]
The underlying meshes we use here are the same as that of Figs. 3.2 and 3.3 in
Hu and Shu \cite{Hu1999}.  The accuracy results of both IQR and SQR schemes are
listed in Tables \ref{tab:cvgBurgersstr} and \ref{tab:cvgBurgersunstr}.  A
third-order accuracy is maintained for both structured and unstructured meshes.

\begin{table}[htbp]
  \caption{Accuracy of 2D Burgers' equation at $t=0.5/\pi^2$ on structured grids}
  \label{tab:cvgBurgersstr}
  \centering
  \begin{tabular}{lrrrrrrrrr}
      \toprule
     \multicolumn{5}{c}{\bf IQR} 
     && \multicolumn{4}{c}{\bf SQR}\\
     \cmidrule(lr){2-5}\cmidrule(lr){6-10}
     $h$ & $ L^1$ error & Order & 
     $ L^\infty$ error & Order &
     & $ L^1$ error & Order & $ L^\infty$ error
     & Order \\ \midrule		
      $2/5$ & 3.38E-01 & --- & 6.63E-02 & --- && 4.16E-01 & --- & 6.72E-02 & --- \\ 
      $1/5$ & 4.54E-02 & 2.89 & 9.89E-03 & 2.75 && 4.54E-02 & 3.20 & 9.89E-03 & 2.76 \\ 
      $1/10$ & 5.73E-03 & 2.99 & 1.31E-03 & 2.92 && 5.73E-03 & 2.99 & 1.31E-03 & 2.92 \\ 
      $1/20$ & 7.21E-04 & 2.99 & 1.67E-04 & 2.98 && 7.21E-04 & 2.99 & 1.67E-04 & 2.98 \\ 
      $1/40$ & 9.14E-05 & 2.98 & 2.09E-05 & 2.99 && 9.01E-05 & 3.00 & 2.09E-05 & 2.99 \\ \bottomrule
  \end{tabular}
\end{table}

\begin{table}[htbp]
  \caption{Accuracy of 2D Burgers' equation at $t=0.5/\pi^2$ on unstructured grids}
  \label{tab:cvgBurgersunstr}
  \centering
  \begin{tabular}{lrrrrrrrrr}
    \toprule
    \multicolumn{5}{c}{\bf IQR} 
    && \multicolumn{4}{c}{\bf SQR}\\
    \cmidrule(lr){2-5}\cmidrule(lr){6-10}
    $h$ & $ L^1$ error & Order & 
    $ L^\infty$ error & Order && $ L^1$ error & Order & $ L^\infty$ error
    & Order \\ \midrule		
    $1/2$ & 2.28E-01 & --- & 7.59E-02 & --- &&  2.27E-01 &  --- &  7.59E-02	 &  --- \\ 
    $1/4$ & 3.02E-02 & 2.92 & 1.18E-02 & 2.68 &&  3.01E-02 &	2.91&	1.18E-02&	2.68 \\ 
    $1/8$ & 3.81E-03 & 2.99 & 1.66E-03 & 2.84  &&  3.80E-03&	2.98	&1.66E-03	&2.84\\ 
    $1/16$ & 4.80E-04 & 2.99 & 2.19E-04 & 2.92 && 4.76E-04	& 3.00	& 2.19E-04	& 2.92 \\ 
    $1/32$ & 6.08E-05 & 2.98 & 2.82E-05 & 2.96 && 5.94E-05&	3.00&	2.82E-05	&2.96 \\ 
    \bottomrule
  \end{tabular}
\end{table}

To demonstrate the application for shock computations we compute until
$t=5/\pi^2$.  Fig. \ref{fig:Burgerssurf} shows the results for a structured mesh
with $h=1/20$ and an unstructured mesh with $h=1/16$, following the resolution
that was used in \cite{Hu1999}.  From here we can see that the shock front,
located on $x+y=3/\pi^2\pm 2$, is captured well. 
\begin{figure}[htbp]
	\subfloat[Structured mesh: $h=1/20$]
  {\includegraphics[width=.5\textwidth] {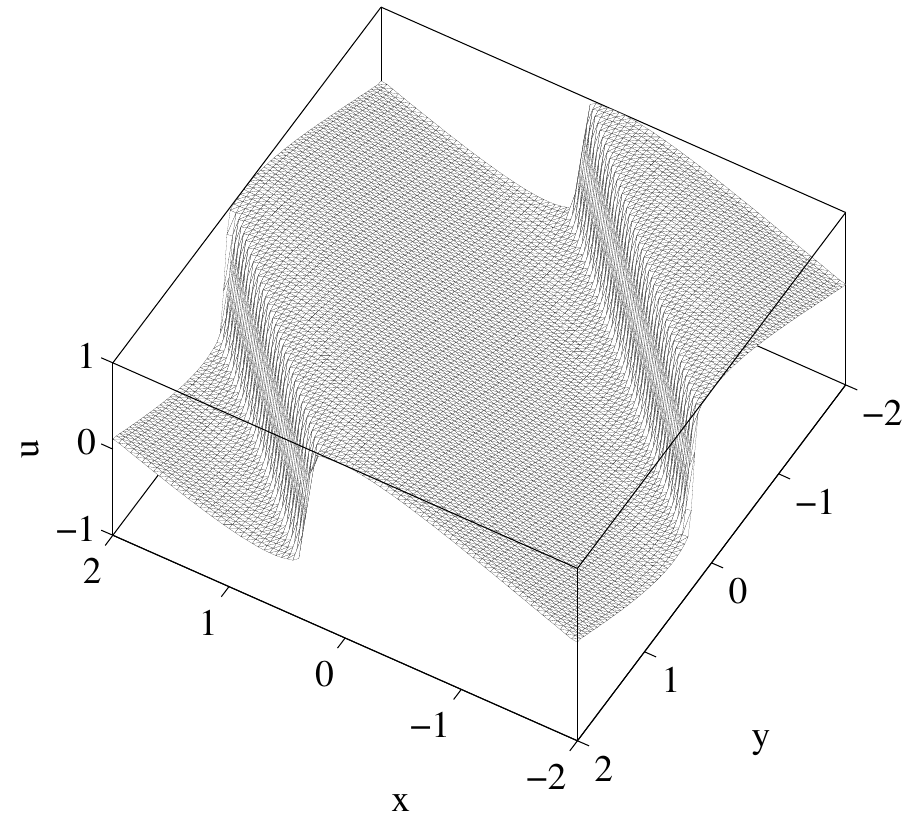}}
	\subfloat[Unstructured mesh: $h=1/16$]
  {\includegraphics[width=.5\textwidth] {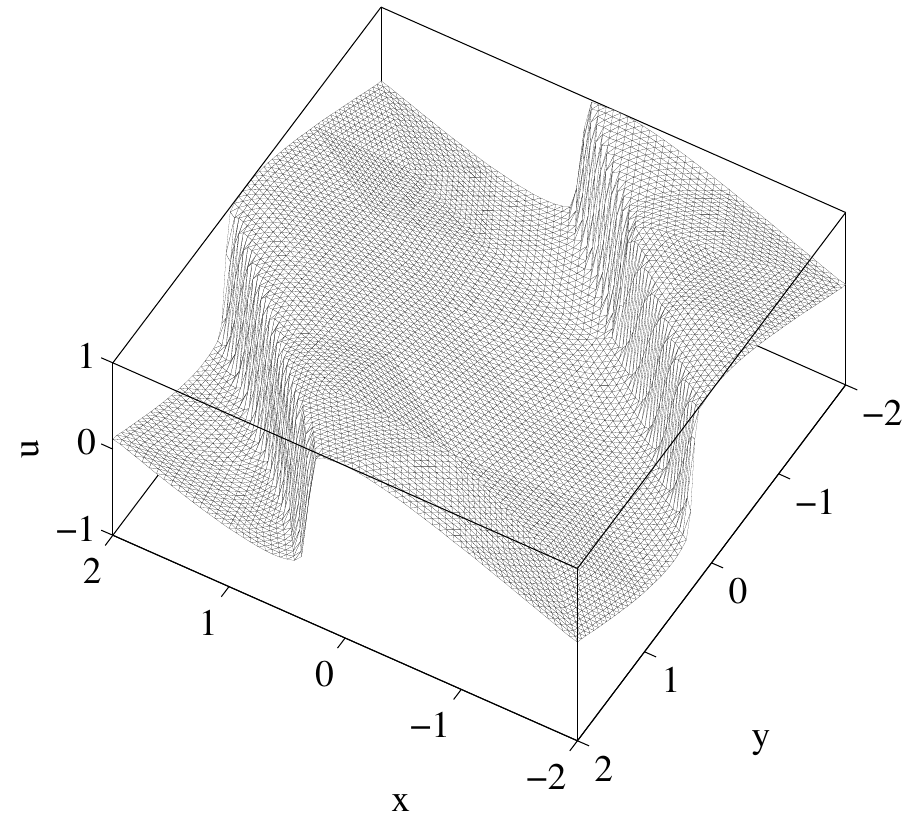}}
	\caption{Solution of 2D Burgers' equation at $t=5/\pi^2$.}
	\label{fig:Burgerssurf}
\end{figure}

\subsection{Two-dimensional Riemann problem}

To investigate the performance of IQR scheme in the presence of genuinely
multi-dimensional nonlinear waves, we solve the Riemann problem
\cite{Christov2008} of Burgers' equations \eqref{eq:burgers}
\[
u(x,y,0) =
\begin{cases}
2, & x,y<0.25,\\
3, & x,y>0.25,\\
1, & \text{Otherwise},
\end{cases}
\]
on the domain $[0,1]\times[0,1]$.  Inflow boundary conditions are prescribed at
the left and bottom edges of the boundary to mimic the motion of shocks.  Two
shock waves and two rarefactions will meet towards the center of the domain to
form a double-parabola-shaped cusp.  In our computation the solution is advanced
to $t=1/12$.  The exact solution can be found by using the method of
characteristics, i.e.
\[
u(x,y,t)=
\begin{cases}
3, & \min\{x,y\}>0.25+3t,\\
(\min\{x,y\}\!-\!0.25)/t, & 0.25+2t-\min\{\sqrt{2|x-y|t},t\}\le \min\{x,y\} 
\le 0.25+3t, \\
1, &  \min\{x,y\}<0.25+t \le 0.25+1.5t < \max\{x,y\}, \\
2, & \text{Otherwise}.
\end{cases}
\]

On successively refined Delaunay meshes, the $L^1$ errors at $t=1/12$
are shown in Fig. \ref{fig:riemann}(a), with an order of accuracy close to one,
which confirms that IQR scheme is genuinely high-order.  The contour lines of
the solution on a fine mesh are shown in Fig. \ref{fig:riemann}(b).  The result
exhibits high resolution for the central cusp and shock front.

\begin{figure}[htbp]
	\centering
	\subfloat[$ L^1$ error versus mesh size (first-order)]
	{\includegraphics[height=58mm]{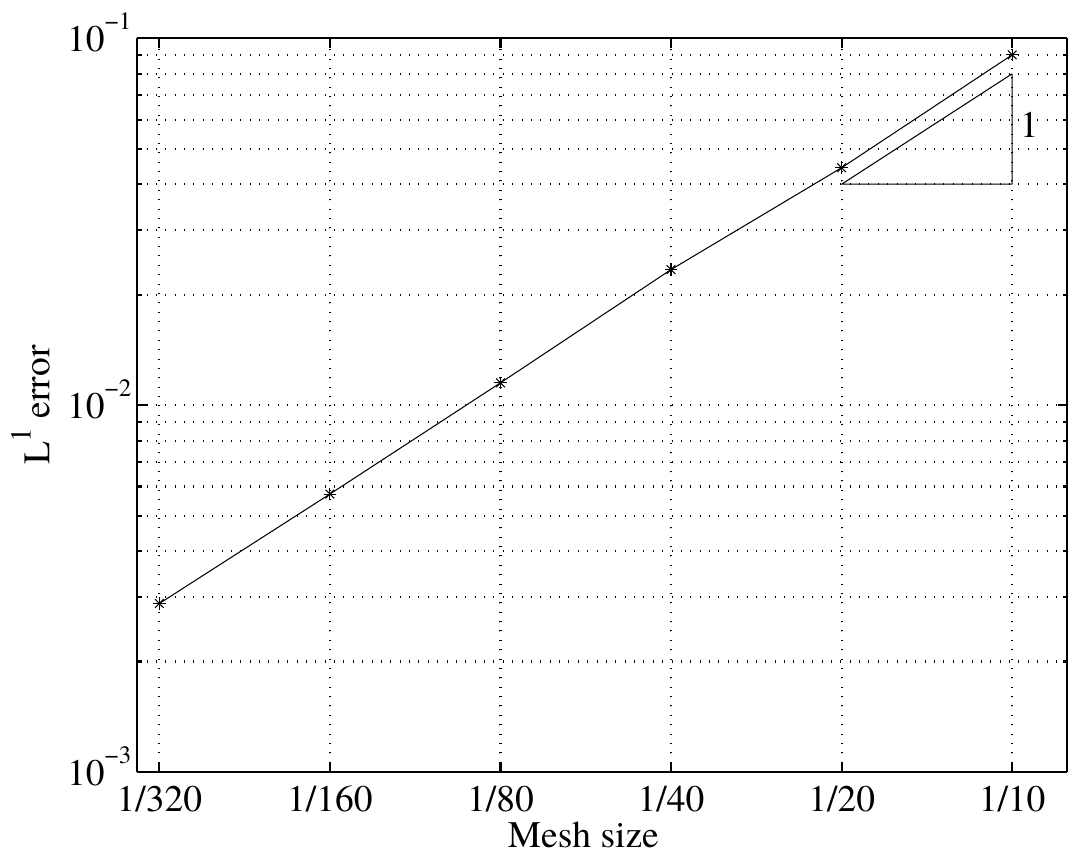}}
	\qquad
  \subfloat[Contour lines on the mesh $h=1/160$]
  {\includegraphics[height=58mm]{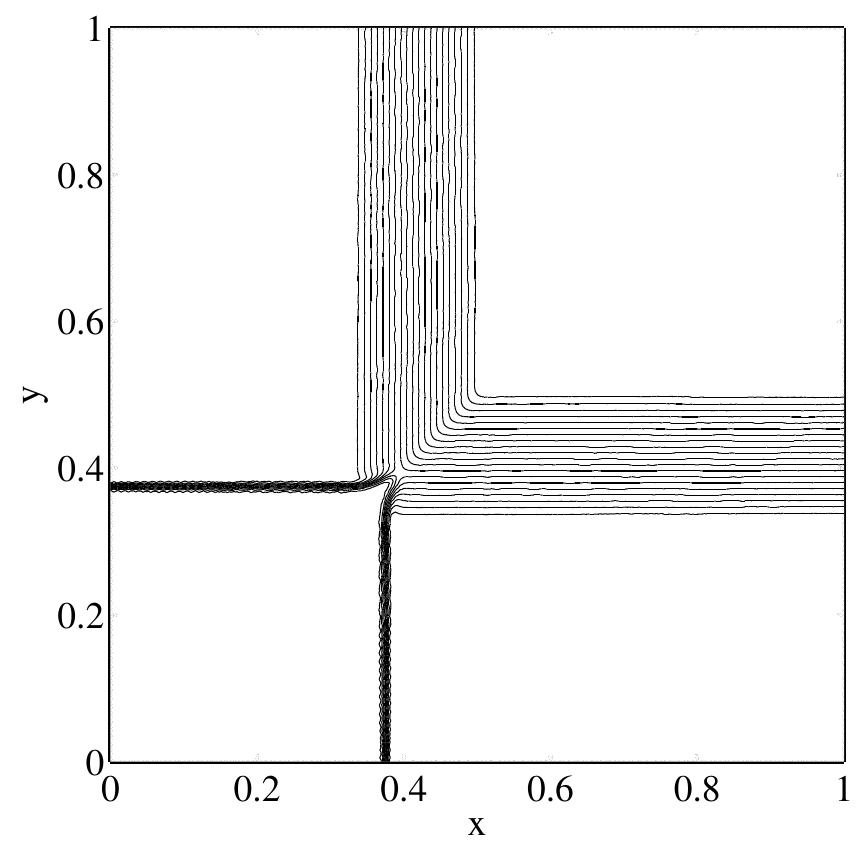}}
	\caption{Riemann problem of Burgers' equation at $t=1/12$.}
	\label{fig:riemann}
\end{figure}

\subsection{Three-dimensional linear equation}
This is a three-dimensional problem used to assess the order of accuracy.  We
solve the following linear equation
\[
u_t+u_x+u_y+u_z=0,
\]
with initial profile given by the triple sine wave function 
\[
u(x,y,z,0)=\sin(2\pi x)\sin(2\pi y)\sin(2\pi z).
\]
The computation domain is $[0,1]\times[0,1]\times[0,1]$. 
We perform the convergence test on unstructured tetrahedron mesh 
with periodic boundary condition applied. In Table \ref{tab:3DLinear}, 
one observes third-order of accuracy on such kind mesh.
\begin{table}[htbp]
  \centering
  \caption{Accuracy for 3D linear equation.}
  \label{tab:3DLinear}
  \begin{tabular}{lrrrr}
    \toprule
    $h$ & $ L^1$ error & Order & 
    $ L^\infty$ error & Order \\ \midrule
    1/10 & 1.05E-01 & --- & 7.59E-01 & --- \\ 
    1/20 & 4.47E-02& 2.03 & 1.86E-01& 2.02 \\ 
    1/40 & 6.69E-03& 2.66 & 2.61E-02& 2.67 \\ 
    1/80 & 8.51E-04& 2.80 & 3.46E-03& 2.75 \\ 
    \bottomrule
  \end{tabular}
\end{table}

\subsection{Three-dimensional Burgers' equation}

In this last test we compute the three-dimensional Burgers' equation \cite{Zhang2009}
\[
		u_t + \left(\dfrac{1}{2}u^2\right)_x + 
			\left(\dfrac{1}{2}u^2\right)_y + 
			\left(\dfrac{1}{2}u^2\right)_z = 0,
\]
with initial data $u(x,y,z,0) = 0.3+0.7\sin({\pi}(x+y+z)/3)$ on the cube domain
$[-3,3]\times[-3,3]\times[-3,3]$ with periodic boundary conditions.  The CFL
number is taken as $0.1$ here.  The convergence order at $t = 0.5/ \pi^2$ is
listed in Table \ref{tab:3DBurgers}, where full accuracy is observed.  We also
present the contour plots of the solution at $t=5/ \pi^2$ on the surface and the
2D slice $z = 0$ as well as the 1D cutting-plot along the line $x = y, z = 0$ in
Fig \ref{fig:Burgers3d}.  We can observe that the solution is non-oscillatory
and the shock is resolved sharply.
\begin{table}[htbp]
  \centering
  \caption{Accuracy for 3D Burgers' equation.}
  \label{tab:3DBurgers}
  \begin{tabular}{lrrrr}
    \toprule
    $h$ & $ L^1$ error & Order & 
    $ L^\infty$ error & Order \\ \midrule
    $3/4$ & 1.42E+01 & --- & 2.88E-01 & --- \\ 
    $3/8$ & 2.17E+00 & 2.71 & 5.32E-02 & 2.44 \\ 
    $3/16$ & 2.85E-01 & 2.93 & 7.61E-03 & 2.81 \\ 
    $3/32$ & 3.59E-02 & 3.00 & 9.95E-04 & 2.94 \\ 
    \bottomrule
  \end{tabular}
\end{table}

\begin{figure}[htbp]
	\subfloat[Solution on the surface]
  {\includegraphics[height=40mm] {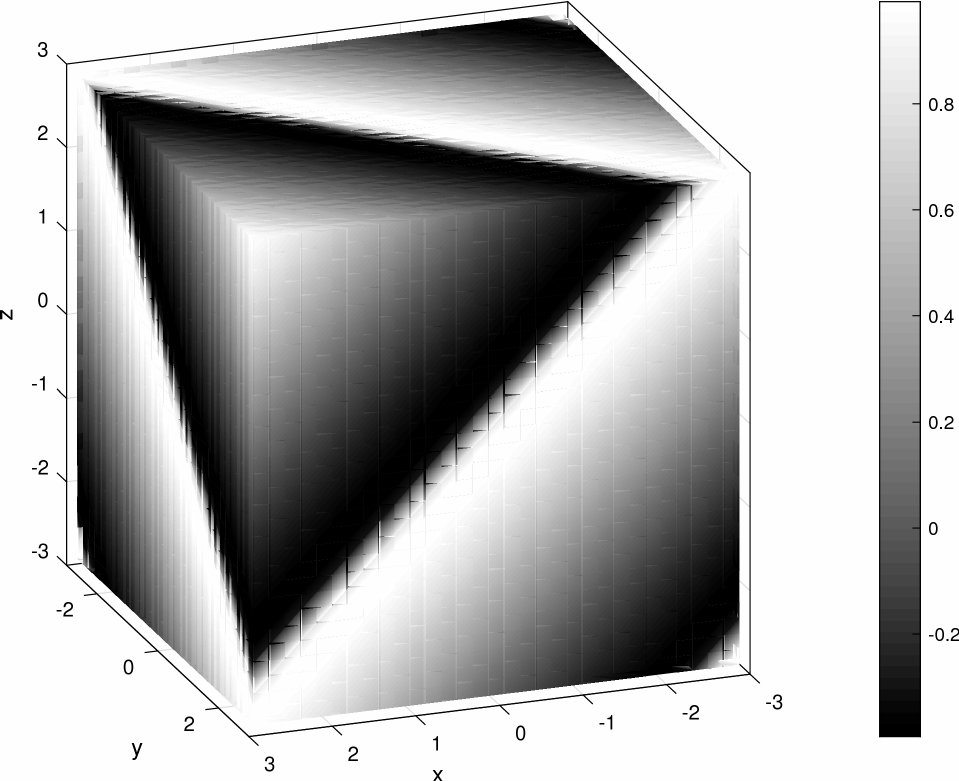}}
	\quad
	\subfloat[Solution on the slice $z=0$]
  {\includegraphics[height=40mm] {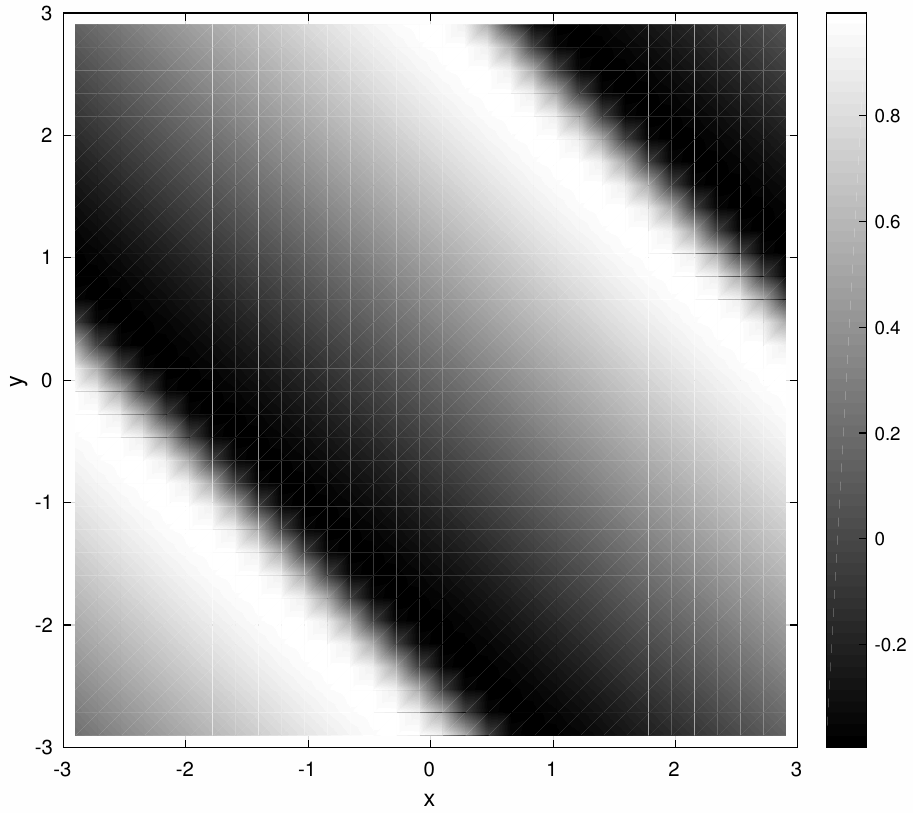}}
	\quad
	\subfloat[Cutting-plot along $x=y$, $z=0$]
  {\includegraphics[height=40mm] {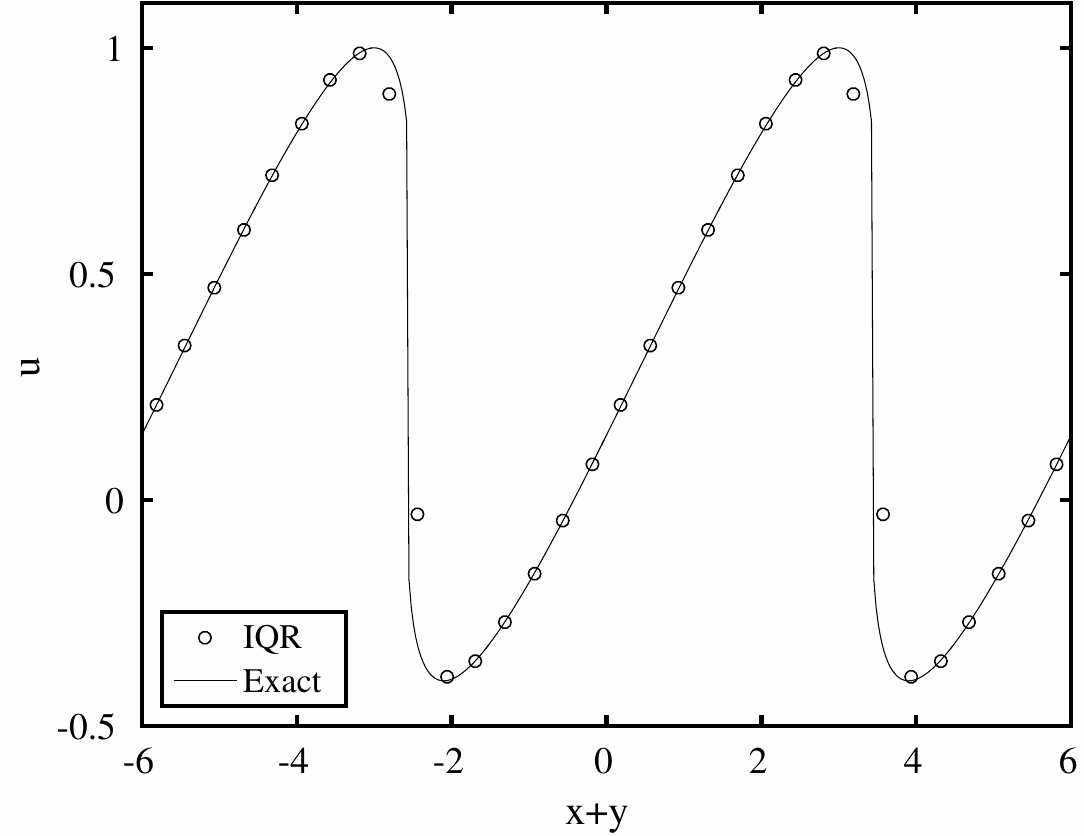}}
	\caption{Three-dimensional Burgers' equation at 
	$t=5/ \pi^2$ with $32\times 32\times 32$ cells.}
	\label{fig:Burgers3d}
\end{figure}

We remark that even a third-order accuracy is observed in the
numerical examples, {\bf Theorem 1} does not guarantee the numerical
solution a third-order accuracy since it depends on if the upper and
the lower bound in the reconstruction have the same accuracy
order. For multiple dimensional problems, it looks the third-order
accuracy able to be preserved during the evolving of the scheme, while
it is harder to preserve the accuracy order for one dimensional
problem since the maximal value can be available at only one point.


\section{Conclusion}

We proposed an integrated quadratic reconstruction method for
high-order finite volume schemes to scalar conservation laws in
multiple dimensions. The reconstruction is applicable on flexible
grids and requires no problem dependent parameters.  Moreover, it
gives us a finite volume scheme which satisfies a local maximum
principle. Numerical results showed the accuracy and robustness of the
proposed scheme. In the future, we are interested in
how to extend the reconstruction strategy here to systems of
conservation laws.


\section*{Acknowledgements}

The authors appreciate the financial supports by 
the Science Challenge Project (No. TZ2016002), 
the National Natural Science Foundation of China (Grant No. 11971041).

\section*{\refname}
\bibliographystyle{unsrt}
\bibliography{refs}
\end{document}